\documentclass[11pt]{article}

\usepackage{amsfonts,amssymb,amsopn,amsmath,mathrsfs,theorem,bbm}
\usepackage[round]{natbib}
\usepackage{hyperref}
\usepackage{graphicx}
\usepackage{verbatim}
\usepackage{authblk}
\usepackage{color}
\usepackage{enumitem}

\newcounter{subsection1}[section]
\newtheorem{defn}[subsection1]{Definition}
\newtheorem{lemma}[subsection1]{Lemma}
\newtheorem{prop}[subsection1]{Proposition}
\newtheorem{theorem}[subsection1]{Theorem}
\newtheorem{remark}[subsection1]{Remark}
\newtheorem{cor}[subsection1]{Corollary}

\newtheorem{notation}[subsection1]{Notation}

\newcounter{assumptions}
\newenvironment{proof}{\vspace{1ex}\noindent{\textsc{Proof:}}\hspace{0.5em}}{\hfill\qed\vspace{1ex}}

\numberwithin{equation}{section} 
\numberwithin{subsection1}{section}

\DeclareMathOperator{\Vote}{\mathbb V}

\begin{document}

\def\comment{\bfseries\textsc}

\def\ra{\Rightarrow} 
\def\to{\rightarrow} 
\def\iff{\Leftrightarrow}
\def\sw{\subseteq} 
\def\mc{\mathcal} 
\def\mb{\mathbb} 
\def\sc{\setminus} 
\def\p{\partial} 
\def\v{\boldsymbol} 
\def\E{\mb{E}} 
\def\P{\mb{P}}
\def\R{\mb{R}} 
\def\C{\mb{C}} 
\def\N{\mb{N}}
\def\Q{\mb{Q}}
\def\Z{\mb{Z}}
\def\IZ{\mb{Z}}
\def\B{\mb{B}}
\def\~{\sim}
\def\-{\,;\,} 
\def\com{\leftrightarrow}
\def\|{\,|\,} 
\def\li{\langle}
\def\ri{\rangle}
\def\wt{\widetilde}
\def\qed{$\blacksquare$}
\def\1{\mathbbm{1}}
\def\cadlag{c\`{a}dl\`{a}g}
\def\slfv{S$\Lambda$FV}
\def\slfvs{SLFVS}
\def\l{\left}
\def\r{\right}

\def\Ca{($\mathscr{C}1$)}
\def\Cb{($\mathscr{C}2$)}
\def\Cc{($\mathscr{C}3$)}
\def\Cd{($\mathscr{C}4$)}
\def\Ce{($\mathscr{C}5$)}

\def\BBMone{\textit{BBM}^1}
\def\BBMtwo{\textit{BBM}^2}
\def\BBM{\textit{BBM}}

\def\comment{\color{blue} }
\def\todo{\comment todo}
\def\fixme{\comment fixme}

\def\epsilon{\varepsilon} 

\def\dim{\mathbbm{d}}

\newcommand{\ts}{\textsuperscript}

\newcommand \ind {\mathbf{1}}

\allowdisplaybreaks

\author[1]{Alison Etheridge\thanks{etheridg@stats.ox.ac.uk, supported in part by EPSRC Grant EP/I01361X/1}}
\author[2]{Nic Freeman\thanks{nicfreeman1209@gmail.com}}
\author[1]{Sarah Penington\thanks{sarah.penington@sjc.ox.ac.uk, supported by EPSRC DTG EP/K503113/1}}
\affil[1]{Department of Statistics, University of Oxford}
\affil[2]{School of Mathematics and Statistics, University of Sheffield}
\title{Branching Brownian Motion, mean curvature flow and the motion of
hybrid zones}

\date{\today}
\maketitle

\begin{abstract}
We provide a probabilistic proof of 
a well known connection between a special case of 
the Allen-Cahn equation and mean curvature flow.
We then prove a
corresponding result for scaling limits of the spatial $\Lambda$-Fleming-Viot process 
with selection, in which the selection mechanism is chosen to model what are known
in population genetics as {\em hybrid zones}. 
Our proofs will exploit a duality with a system of branching (and coalescing) 
random walkers which is of some interest in its own right. 
\end{abstract}

\tableofcontents

\section{Introduction}

Our central result, Theorem~\ref{thm:slfvs} in Section~\ref{sec:slfvs_intro},
is the convergence, after suitable rescaling, of a 
stochastic analogue of the Allen-Cahn equation to the indicator function of
a region whose boundary evolves according to mean curvature flow.
The main motivation for this work comes from mathematical population genetics;
specifically,
we are interested in the behaviour of so-called hybrid zones. These occur when genetically distinct groups 
of individuals meet and mate, leaving behind at least some offspring of mixed ancestry. A textbook 
example is the common house mouse in Denmark \citep{hunt/selander:1973} which exists in the
form {\em Mus musculus} in the North and {\em M.~domesticus} in the South, but hybrid zones are
ubiquitous in nature, for example, \cite{barton/hewitt:1989} cite 170 examples. 
Two principal explanations have been offered for the genetic variation observed in such zones.
The first is that they arise in response to spatially varying natural selection; the second is 
that they are formed through secondary contact of two populations that were previously genetically
isolated. Whereas in the first scenario the location of the hybrid zone
is determined by an environment, which is usually taken to be fixed, in 
the second scenario, the hybrid zone can evolve with time.
It is this second scenario that interests us here. 

It is usual to suppose that the underlying genetics is controlled by a single gene which occurs
in two types (alleles), traditionally denoted $a$ and $A$. Individuals carry two copies of the gene and
while those of types $aa$ and $AA$ (the {\em homozygotes}) are equally fit, the {\em heterozygotes} (that is individuals of 
type $aA$) are less likely to successfully reproduce. 
In an infinitely dense
population, provided the selection against heterozygotes is weak, when 
viewed over large spatial and temporal scales, the proportion of
$a$-alleles in the population at location $x$ at time $t$ is modelled by the solution
to
\begin{equation}
\label{AC1}
\frac{\partial v}{\partial t}=\Delta v+\v{s} v(1-v)(2v-1),
\end{equation}
for an appropriate initial condition, where $\v{s}>0$ is a scaled selection coefficient. This is
a special case of the Allen-Cahn equation; we explain the
origin of this particular form of nonlinearity in 
Section~\ref{slfvs for hybrid zones}.

Our interest is in the behaviour of the region in which both alleles are present in substantial numbers.
Because heterozygotes are less fit than homozygotes, we expect this to be a narrow band which, when 
viewed on large enough scales, will look like a sharp interface. More formally, we apply a diffusive
scaling to~(\ref{AC1}) in which $t\mapsto \epsilon^2t$ and $x\mapsto \epsilon x$. 
The Laplacian term is, of course, invariant, but the term corresponding to 
selection is multiplied by a factor $1/\epsilon^2$.
It is well known that for suitable initial conditions, in a sense that we make precise in 
Theorem~\ref{theorem ac to cf}, as $\epsilon\downarrow 0$, the solution to the
scaled equation converges to the indicator function of a set whose boundary evolves according to 
mean curvature flow.
Thus, in the biologically relevant case of two dimensions, if we observe
the population over sufficiently large spatial and temporal scales, the interface between the
two populations will evolve approximately as curvature flow or {\em curve-shortening flow} as it is
often known.

One reason for the importance of curvature flow in applications stems from an 
underlying variational principle: curve shortening flow decreases the length of the 
curve at the fastest rate possible relative to the total speed of motion (measured in 
the sense of the square integral of the speed of motion of points around
the curve), see e.g.~\cite{white:2002} for a simple explanation. 
In this sense, if our populations evolved deterministically, 
then they would
minimise the boundary between them as quickly as possible. 
In reality this will be somewhat offset
by the randomness due to reproduction, known as random genetic drift,
in a population which is not infinitely dense. 
Indeed if genetic drift is too strong, then we can expect the random noise
to obscure the nonlinear term: this is suggested by the results of 
\cite{hairer/ryser/weber:2012}, who consider the equation
$$dw=(\Delta w +w-w^3)dt +\sigma dW,$$
in two dimensions, where $W$ is a mollified space-time white noise. 
(By considering $(1+w)/2$, up to constants, we recover a stochastic 
version of~(\ref{AC1}).) If the mollifier
is removed, then the solutions converge weakly to zero, whereas if the intensity
of the noise simultaneously converges to zero sufficiently quickly, then
they recover the deterministic equation. 
The basic question that we set out to answer is ``Will hybrid zones still evolve approximately
according to curvature flow in the presence of random genetic drift?''

Of course, genetic drift is not appropriately modelled by a mollified
space-time white noise and so,
in order to investigate this question, we must first
define a model that combines selection against heterozygosity
with random genetic drift. Our starting point will be the spatial 
$\Lambda$-Fleming-Viot 
process which was introduced in \cite{etheridge:2008, barton/etheridge/veber:2010} and has been
studied in a series of papers since; 
see e.g.~\cite{barton/etheridge/veber:2013} for a review. The
advantage of this model is that it allows us to incorporate genetic 
drift into models of populations 
evolving in spatial continua, with no restriction on spatial dimension. 
However, since our proofs are based on a duality with a branching and
coalescing random walk, we expect analogous results if we start, for
example, from the classical stepping stone model in which the population
is subdivided into `islands' that sit at the vertices of $\IZ^\dim$.
In what follows, 
we shall refer to the spatial $\Lambda$-Fleming-Viot process with 
selection against heterozygosity as the
{\slfvs}. It is described carefully in Definition~\ref{slfvdefn}.
A version of this model with 
selection in favour of one genetic type was constructed in \cite{etheridge/veber/yu:2014}. There it
was shown that when suitably rescaled, in two or more dimensions, 
the allele frequencies converge to a solution of the Fisher-KPP equation,
\begin{equation}
\label{fisher KPP}
\frac{\partial v}{\partial t}=\Delta v+\v{s} v(1-v).
\end{equation}
Mimicking that result, one
can obtain~(\ref{AC1}) as a scaling limit of the {\slfvs}. Combined 
with the known convergence
of the scaled version of~(\ref{AC1}),
this certainly suggests that there should be scalings of the {\slfvs} which lead to mean curvature flow.
However, available proofs of Theorem~\ref{theorem ac to cf} could not 
readily be adapted to our 
stochastic setting and so we were forced to seek an alternative approach. 
Our first result is therefore a new proof of Theorem~\ref{theorem ac to cf}.
We then adapt this to prove convergence of the proportions of 
different genetic types under the {\slfvs} to 
the indicator function of a set whose boundary evolves according
to mean curvature flow.
The key to our proof
is a probabilistic representation of solutions to~(\ref{AC1}) which we believe to be of interest in
its own right. 

Before defining the {\slfvs}, we recall some purely deterministic results. Although our primary interest
is in two spatial dimensions, there will be no additional arguments required if we work in $\R^{\dim}$ for
arbitrary $\dim >1$.

\subsection{The Allen-Cahn equation and mean curvature flow}

The Allen-Cahn equation \citep{allen/cahn:1979} takes the form
\begin{equation}
\label{allen-cahn equation}
\frac{\partial v^\epsilon}{\partial t}=\Delta v^\epsilon -\frac{1}{\epsilon^2}f(v^\epsilon),
\end{equation}
where $f$ is the derivative of a potential function $F$ which has 
exactly two local minima, at $v_-$ and
$v_+$, say.  More precisely, we insist that $f\in C^2(\R)$ has exactly three zeros, $v_-<v_0<v_+$, and
\begin{equation}
\label{conditions on potential}
\begin{array}{ll}
f(v)<0, & \forall v\in (-\infty, v_{-})\cup (v_0, v_+);\\
f(v)>0, & \forall v\in (v_-,v_0)\cup (v_+,\infty);\\
f'(v_-)>0, & f'(v_+)>0, \quad f'(v_0)<0.
\end{array}
\end{equation}
Although originally introduced as a model for the macroscopic motion of phase boundaries
driven by surface tension, the Allen-Cahn equation has found application in many other areas. 
It represents a balance between two opposing tendencies: the diffusive effect of the 
Laplacian attempts to smooth the solution, while the potential term drives it towards
the states $v_-$ and $v_+$. As a result, a narrow interface between these two states develops.

Allen and Cahn observed that if the two potential wells do not have 
equal depth, then on
the timescale $s=t/\epsilon$, the interface will propagate at a constant 
speed (proportional to $F(v_-)-F(v_+)$) along its normal, towards the domain
of the deeper well. On the other hand, if the potential wells have equal depth, then the 
interface is almost stationary on this timescale, but if we observe it
over the longer  
timescales of~(\ref{allen-cahn equation}), it will propagate with normal 
velocity equal to the 
mean curvature of the interface. 

There is now a huge literature that makes the observation of Allen and Cahn
rigorous under various regularity conditions, for example
\cite{bronsard/kohn:1991, evans/soner/souganidis:1992, 
ilmanen:1993, sato:2008}.
The principal obstruction to be overcome relates to the fact that the mean curvature flow 
is only well-defined under some regularity conditions and, even then, only up to a finite time
horizon when it either shrinks to a point or, in dimensions three and higher, develops other
singularities.

Before stating a result, let us make the definition of mean curvature flow precise.
We begin with the special case of two dimensions. This is the relevant dimension for 
our biological application and requires much less explanation. In that setting,
mean curvature is just curvature and the corresponding flow is often called
curve-shortening.

Recall that a function is said to be a smooth embedding 
if it is a diffeomorphism onto its image (which we shall implicitly assume 
is a subset of $\R^2$).
\begin{defn}[Curve-shortening flow]
\label{def:curvatureflow}
Let $S^1$ denote the unit circle in $\R^2$. Let $\v{\Gamma}=(\v{\Gamma}_t(\cdot))_t$ be a family of smooth embeddings, indexed by $t\in[0,\mathscr{T})$, where for each $t$, $\v{\Gamma}_t:S^1\to\R^2$. 
Let $\v{n}=\v{n}_t(\phi)$ denote the unit (inward) normal vector to $\v{\Gamma}_t$ at $\phi$ and let $\kappa=\kappa_t(\phi)$ denote the curvature of $\v{\Gamma}_t$ at $\phi$.
We say that $\v{\Gamma}$ is a {\em curvature flow} or {\em curve-shortening flow} if
\begin{equation}\label{eq:cf_pre}
\frac{\p \v{\Gamma}_t(\phi)}{\p t}=\kappa _t(\phi)\v{n}_t(\phi).
\end{equation}
for all $t,\phi$.
\end{defn}
Assuming that $\v{\Gamma}_0$ is a smooth embedding of $S^1$ into $\R^2$,
the behaviour of $\v{\Gamma}_t$ under curve-shortening is 
completely understood. First, it has a finite lifetime which we
shall denote by $\mathscr{T}$. In
\cite{gage/hamilton:1986}, it was shown that if $\v{\Gamma}_0$ is convex, then so is
$\v{\Gamma}_t$ for all $t<\mathscr{T}$. Moreover, $\mathscr{T}$ can be chosen so that $\v{\Gamma}_t$ shrinks towards a point as $t\uparrow\mathscr{T}$; in this limit the asymptotic `shape' of $\v{\Gamma}_t$ is a circle. Soon afterwards, \cite{grayson:1987} showed that,
in fact, under curve-shortening, any smoothly embedded closed curve becomes convex at a
time $\tau<\mathscr{T}$, after which the results of \citeauthor{gage/hamilton:1986} apply.

In higher dimensions we must replace the curvature by the {\em mean curvature}.
Recall that to define this quantity for a $(\dim -1)$-dimensional
hypersurface in $\R^{\dim}$, we take an orthonormal basis of the 
tangent space and form the matrix of the second fundamental form, that is 
the matrix whose $(i,j)$th entry is 
the dot product of the unit normal to the hypersurface with 
the derivative of the $i$th 
vector in the basis in the direction of the $j$th.
The 
$\dim-1$ principal curvatures, $\kappa_1,\ldots ,\kappa_{\dim-1}$, 
are the eigenvalues of the matrix and their sum, 
that is the trace of the matrix,
is the (scalar) mean curvature.
The product of the scalar mean curvature with the unit normal is
called the mean curvature vector (which does not 
depend on the choice of normal, since reversing
the direction of the normal also changes the sign of the scalar mean curvature).

\begin{defn}[Mean curvature flow]
{\em Mean curvature flow}, when it is defined,
is obtained by replacing the curvature $\kappa_t$ in
equation~(\ref{eq:cf_pre}) by the mean curvature.
\end{defn}
The behaviour of mean curvature flow in $\dim\geq 3$ is more complex than that of curve-shortening.
It was proved by~\cite{huisken:1984} that the analogue of the Gage-Hamilton Theorem holds,
that is a $(\dim-1)$--dimensional compact convex surface must shrink to a point and its asymptotic
shape is a sphere. However, the analogue of Grayson's Theorem is false. In higher dimensions
singularities can develop before the enclosed volume vanishes. Since our main interest is in
two dimensions, we shall not discuss this here. Instead we shall follow~\cite{chen:1992} in
imposing sufficiently strong initial conditions that the solution exists for a positive time and
stopping before we encounter any singularities, and we refer to~\cite{mantegazza:2011}
for a detailed discussion.

Suppose that $\dim\geq 2$.
Our first result concerns the convergence as $\epsilon\downarrow 0$, 
for suitable initial conditions, of the solution of
\begin{equation}
\label{special allen-cahn equation}
\frac{\partial v^\epsilon}{\partial t}=\Delta v^\epsilon 
+\frac{1}{\epsilon^2}v^\epsilon (1- v^\epsilon)(2 v^{\epsilon} -1),\qquad v^\epsilon(0, x)=p(x),
\end{equation}
to the indicator function of a set whose boundary evolves according to 
mean curvature flow.

The initial condition, $p$, of~(\ref{special allen-cahn equation})
is  assumed to take values in $[0,1]$. We shall also require that
it satisfies some regularity conditions. In particular, set
$$\Gamma=\l\{x\in\R^\dim : p(x)=\frac{1}{2}\r\}.$$
We suppose that $\Gamma$ is a smooth hypersurface which is also the boundary
of a bounded open set which is topologically equivalent to the sphere. 
We impose the following regularity conditions:
\begin{enumerate}[leftmargin=1.1cm]
\item[{\Ca}] $\Gamma$ is $C^{\alpha}$ for some $\alpha>3$.
\item[{\Cb}] For $x$ inside $\Gamma$, $p(x)<\tfrac{1}{2}$. For $x$ outside $\Gamma$, $p(x)>\tfrac{1}{2}$.
\item[{\Cc}] There exist $r,\gamma>0$ such that, for all $x\in\R^{\dim}$, 
$|p(x)-\frac{1}{2}|\geq \gamma\,\big(\text{dist}(x,\Gamma)\wedge r\big)$.
\end{enumerate}
In particular, we can think of $\Gamma$ as the image of the boundary of 
the unit sphere under a map $f$ for which $|f(x)-f(y)|=\mc{O}(|x-y|^{\alpha})$.
Condition {\Cc} prevents the slope of $p$ near the interface $\Gamma$ 
from being too shallow, and keeps $p(x)$ bounded away from $\frac{1}{2}$ 
when $x$ is not near the interface. 
Condition {\Cb} is simply establishing a sign convention.  
Under these conditions, mean curvature flow started from $\Gamma$, which we denote $(\v{\Gamma}_t(\cdot))_t$, exists
up to some finite time $\mathscr{T}$ (e.g.~\cite{evans/spruck:1991}).

To give a precise statement of the result, we require some more notation.
Let $d(x,t)$ be the signed distance from $x$ to $\v{\Gamma}_t$, 
chosen to be negative 
inside $\v{\Gamma}_t$ and positive outside. Note that, as sets,
$$\v{\Gamma}_t=\{x\in\R^\dim : d(x,t)=0\}.$$

\begin{theorem} \label{theorem ac to cf}
Let $v^\epsilon$ 
solve~(\ref{special allen-cahn equation}) with initial condition $p$
satisfying the conditions {\Ca}-{\Cc}, and define $\mathscr{T}$, $d(x,t)$ as above. 
Fix $T^*\in (0,\mathscr{T})$. 
Let $k\in\N$. There exists 
$\epsilon_\dim(k)>0$, and $a_\dim(k),c_\dim(k)\in(0,\infty)$ such 
that for all $\epsilon\in(0,\epsilon_\dim)$ and $t$ satisfying 
$a_\dim\epsilon ^2 |\log \epsilon |\leq t\leq T^*,$
\begin{enumerate}
\item for $x$ such that $d(x,t)\geq c_\dim\epsilon |\log \epsilon|$, we have 
$v^\epsilon (t,x)\geq 1-\epsilon^k$;
\item for $x$ such that $d(x,t)\leq -c_\dim\epsilon |\log \epsilon|$, we have 
$v^\epsilon (t,x)\leq \epsilon^k$.
\end{enumerate}
\end{theorem}
This result is not new; it is a special case of Theorem~3 of \cite{chen:1992}. 
Indeed, our
proof will display the same key steps: 
first we show that an interface develops; second 
we show that this interface propagates according to (mean) curvature flow.
To achieve the second step, we couple the distance between a
$\dim$-dimensional Brownian
motion and the interface $\Gamma_s$ with a one-dimensional Brownian motion. 
This parallels the approximation of the solution to the Allen-Cahn equation 
by a one-dimensional standing wave in the proof of \cite{chen:1992} (although
we remark that we achieve our coupling through a different perturbation of
the potential than that used by \cite{chen:1992}).
Both steps of our proof use
probabilistic arguments, exploiting a duality 
between solutions to~(\ref{special allen-cahn equation})
and a branching Brownian motion, which is of some interest
in its own right.

\subsection{Modelling hybrid zones}
\label{slfvs for hybrid zones}

Let us now turn to our model of hybrid zones.
Our
starting point is the spatial $\Lambda$-Fleming-Viot process with selection.
The model we consider here is a modification of that introduced for 
genic selection (selection
in favour of just one of the alleles) in \cite{etheridge/veber/yu:2014}, 
and existence of 
the process follows by the same arguments. 
Also as for genic selection, uniqueness follows 
from duality with a system of branching and coalescing particles,
although there is a slight twist in the form that duality takes 
(see Section~\ref{duality for SLFVS}),
mirroring our probabilistic
representation of solutions to~(\ref{AC1}).

We suppose that there are two alleles, $a$ and $A$.
At each time $t$, the 
random function $\{w_t(x),\, x\in \R^\dim\}$ 
is defined, up to a Lebesgue null set of $\R^\dim$, by
\begin{equation}
\label{defn of w}
w_t(x):= \hbox{ proportion of type }a\hbox{ at spatial position }x\hbox{ at time }t.
\end{equation}
In other words, if we sample an allele from the point $x$ at time $t$,
the probability that it is of type $a$ is $w_t(x)$.
\begin{remark}
It is convenient to extend the definition of $w_t(x)$ to all of $\R^{\dim}$ and
so,
on the Lebesgue null set on which~(\ref{defn of w}) is not sufficient to 
specify $w_t(x)$, we shall arbitrarily impose $w_t(x)=0$.
\end{remark}
A construction of an appropriate state space for 
$x\mapsto w_t(x)$ can be found in \cite{veber/wakolbinger:2015}. 
Using the identification
$$
\int_{\R^d} \big\{w(x)f(x,a)+ (1-w(x))f(x,A)\big\}\, dx=\int_{\R^d\times \{a,A\}} f(x,\kappa) M(dx,d\kappa),
$$
this state space is in one-to-one correspondence with the space
${\cal M}_\lambda$ of measures on $\R^\dim \times\{a,A\}$ with `spatial marginal' Lebesgue measure,
which we endow with the topology of vague convergence. By a slight abuse of notation, we also denote the
state space of the process $(w_t)_{t\in\R}$ by 
${\cal M}_\lambda$.

\begin{defn}[Spatial $\Lambda$-Fleming-Viot with selection 
against heterozygosity ({\slfvs})]
\label{slfvdefn}
Fix $u\in (0,1]$ and $\mc{R}\in(0,\infty)$.  
Let $\mu$ be a finite measure on $(0,\mc{R}]$.
Further, let $\Pi$ be a Poisson point process on 
$\R_+\times \R^\dim \times (0,\mc{R}]$ 
with intensity measure 
\begin{equation}\label{slfvdrive}
dt\otimes dx\otimes \mu(dr). 
\end{equation}
The {\em spatial $\Lambda$-Fleming-Viot process with selection} (SLFVS) 
driven by $\Pi$ is the ${\cal M}_\lambda$-valued process 
$(w_t)_{t\geq 0}$ with dynamics given as follows.

If $(t,x,r)\in \Pi$, a reproduction event occurs at time $t$ within the 
closed ball $\mc{B}_r(x)$ of radius $r$ centred on $x$. 
With probability $1-\v{s}$ the event is {\em neutral}, in which case:
\begin{enumerate}
\item Choose a parental location $z$ uniformly at random within 
$\mc{B}_r(x)$, 
and a parental type, $\alpha_0$, according to $w_{t-}(z)$, that is
$\alpha_0=a$ with probability $w_{t-}(z)$ 
and $\alpha_0=A$ with probability $1-w_{t-}(z)$.
\item For every $y\in \mc{B}_r(x)$, set 
$w_t(y) = (1-u)w_{t-}(y) + u\ind_{\{\alpha_0=a\}}$.
\end{enumerate}
With the complementary probability $\v{s}$ the event is {\em selective},
in which case:
\begin{enumerate}
\item Choose three `potential' parental locations $z_1, z_2, z_3$ 
independently and uniformly at random within $\mc{B}_r(x)$, and at each 
of these sites `potential' parental types 
$\alpha_1$, $\alpha_2$, $\alpha_3$ according to $w_{t-}(z_1), w_{t-}(z_2), w_{t-}(z_3)$ respectively.
Let $\widehat{\alpha}$ denote the most common allelic type in 
$\alpha_1,\alpha_2,\alpha_3$.
\item For every $y\in \mc{B}_r(x)$ set 
$w_t(y) = (1-u)w_{t-}(y) + u\ind_{\{\widehat{\alpha}=a\}}$.
\end{enumerate}
\end{defn}
\begin{remark}
More generally, the parameter $u$, which we shall refer to as the 
{\em impact}, can be taken to be random.
In this case, for each $r\in (0,\mc{R}]$,
we let $\nu_r$ be a probability measure on $(0,1]$ and 
the driving noise, $\Pi$, is taken to be a Poisson point process on 
$\R_+\times \R^\dim \times (0,\mc{R}]\times (0,1]$ 
with intensity measure 
\begin{equation*}
dt\otimes dx\otimes \mu(dr)\nu_r(du).
\end{equation*}
For each point $(t,x,r,u)\in\Pi$, the corresponding reproduction event is
described exactly as before.
\end{remark}
Since $\v{s}$ is assumed small, as one expects in a model of genetic drift, to first order the variance of the
increment of the mean allele frequency in the region affected by an event is 
$u^2\bar{w}(1-\bar{w})$, where $\bar{w}$ is the mean of $w_{t-}$ over the
affected region.
Let us try to motivate the form of the selection mechanism, which is
what drives the expectation of the increments in allele frequencies. 
As is usual in population genetics, we have approximated a model
of selection acting on a diploid population (in which each individual
carries two copies of the gene) by one in which we think of selection
acting on single copies of the gene, but in a way that depends on the
local frequencies of the different alleles. This sort of approximation,
which goes back at least to \cite{fisher:1937},
is valid when the local population size is large, corresponding in our case
to the impact $u$ being small. (In fact we are interested in limits in which
the impact will tend to zero.) The idea is simple.
Each individual in the population
carries two copies of the gene. This subdivides the population into 
{\em homozygotes}, carrying either $aa$ or $AA$ and assumed equally fit,
and {\em heterozygotes} carrying $aA$ and assumed to have relative fitness 
$1-\v{s}$. The population is assumed to be in Hardy-Weinberg proportions,
so that if the proportion of $a$-alleles in the parental population is
$\bar{w}$, then the proportions of parents that are of type $aa$, $aA$ and 
$AA$ are $\bar{w}^2$, $2\bar{w}(1-\bar{w})$ and $(1-\bar{w})^2$,
respectively.
During reproduction, each individual produces a very 
large number of germ cells (cells of the same genotype). 
To reflect the relative fitnesses, a heterozygote produces 
$(1-\v{s})$ times as many germ cells as a homozygote.
Germ cells then split into an effectively infinite pool of gametes
(cells containing just one chromosome from each pair) which fuse at
random to form diploid offspring.
Suppose that the proportion of 
type $a$ alleles in the affected region
immediately before reproduction is $\bar{w}$. 
Then the probability that a gamete sampled from the pool is of type $a$ is
\begin{eqnarray}
\frac{\bar{w}^2+\bar{w}(1-\bar{w})(1-\v{s})}{1-2\v{s}\bar{w}(1-\bar{w})}
&=&(1-\v{s})\bar{w}+\v{s}(3\bar{w}^2-2\bar{w}^3)+{\mathcal O}(\v{s}^2)
\nonumber\\
&=& (1-\v{s})\bar{w}+\v{s}(\bar{w}^3+3\bar{w}^2(1-\bar{w}))+{\mathcal O}(\v{s}^2).
\label{selection mechanism}
\end{eqnarray}
Notice that the first term in~(\ref{selection mechanism})
is $1-\v{s}$ times the probability that
an allele sampled from the parental population is of type $a$ whereas the
second is $\v{s}$ times the probability that the majority of three 
alleles sampled independently
from the parental population are of type $a$. This then
motivates the two types of event in our SLFVS. 
In particular,
if we replace a proportion $u$ of the population by offspring, then 
the expected increment in $\bar{w}$ is 
$$u \v{s}(\bar{w}^3+3\bar{w}^2(1-\bar{w})-\bar{w})
=u\v{s}\bar{w}(1-\bar{w})(2\bar{w}-1),$$
which underpins the connection to~(\ref{AC1}).

Of course, in replacing a diploid model by one based directly on allele
frequencies, we have rather muddied the notion of parent in our reproduction
mechanism, 
so the use of the term in Definition~\ref{slfvdefn}
should not be interpreted too literally.

\subsection{Convergence of the hybrid zone to mean curvature flow}
\label{sec:slfvs_intro}

To understand our main result, first we state a
simple modification of a
result on a rescaling of the SLFVS from
\cite{etheridge/veber/yu:2014}.
To state that result, we specialise to 
$\mu(dr)= \delta_R(dr)$, for some fixed $R>0$. 
At the $n$th stage of the rescaling, the impact and selection parameters
are assumed to satisfy
$$
u_n = \frac{u}{n^{1-2\beta}}, \qquad \mbox{and} 
\qquad \v{s}_n=\frac{\rho}{n^{2\beta}}.
$$
Next, we define the averaged process,
$$
w^n_t(x) := w_{nt}(n^{\beta}x), \qquad \hbox{and}\qquad 
\bar{w}^n_t(x):=\frac{n^{\beta \dim}}{V_R}\, \int_{B(x,n^{-\beta}R)}w^n_t(y)\, dy,
$$
where $V_R$ is the volume of the ball of radius $R$ in $\R^{\dim}$.
To simplify notation, we write ${\mathcal M}$ for ${\mathcal M}_\lambda(\R^{\dim} \times\{a,A\})$,
and $D_{\mathcal M}[0,\infty)$ for the set of all 
c\`adl\`ag paths with values in
${\mathcal M}$. We also write $C_c^\infty(\R^{\dim})$ for the set of smooth
compactly supported functions on $\R^{\dim}$.
\begin{theorem}\label{th:evy}[Modification of Theorem~1.3 of~\cite{etheridge/veber/yu:2014}]
Suppose that 
$\beta\in (0,1/3)$, and that
$\bar{w}^n_0$ converges weakly to some 
$w^0\in {\mathcal M}$.  Then, as $n\rightarrow \infty$, the process
$(\bar{w}_t^n)_{t\geq 0}$
converges weakly in $D_{\mathcal M}[0,\infty)$ towards a process
$(w_t^\infty)_{t\geq 0}$
with initial value
$w^{\infty}_0=w^0$. Furthermore,
$(w_t^\infty)_{t\geq 0}$
is the unique deterministic process
for which, for every $f\in C^{\infty}_c(\R^{\dim})$,
$$
\langle w^\infty_t,f\rangle = \langle w_0^\infty,f\rangle +
\int_0^t \bigg\{\frac{\kappa_R}{2}\, \langle w_s^\infty ,\Delta f\rangle 
+ u\rho V_R\, \langle w_s^\infty (1-w_s^\infty) (2w_s^\infty-1)
,f\rangle\bigg\}\, ds,
$$
where
\begin{equation}\label{def Gamma}
\kappa_R= 
\frac{u}{V_R}\int_{B(0,R)}\int_{B(x,R)}(z_1)^2dz\, dx
\end{equation}
with 
$z_1$ the
first coordinate of the vector $z\in\R^{\dim}$. In particular, $\kappa_R$
depends only on $R$ and $\dim$.
\end{theorem}
In other words, up to a change of coefficients,
$(w_t^\infty)_{t\geq 0}$ is a weak solution 
of~(\ref{special allen-cahn equation}) 
with $w_0=w^0$.
Based on Theorem~\ref{th:evy}, 
it is natural to ask whether we can modify the scaling of $\v{s}_n$ in such a 
way that $\v{s}_nn^{2\beta}\rightarrow\infty$ as $n\rightarrow\infty$ and 
obtain convergence to the indicator function of a region whose boundary 
evolves according to mean curvature flow.
In other words, does genetic drift, 
which is driven by the neutral events in the SLFVS, disrupt that 
convergence?

To state our result, 
we first rescale the {\slfvs} as in Theorem~\ref{th:evy}.
For each $n\in\N$, we define the finite measure $\mu^n$ on $(0, \mathcal R_n]$,
where $\mathcal R_n = n^{-\beta}\mathcal R$, 
by $\mu^n(A)=\mu(n^{\beta}A)$
for all Borel subsets $A$ of $(0,\infty)$.
Our rescaled {\slfvs} will be driven by the Poisson point process $\Pi^n$ on 
$\R_+ \times \R^{\dim} \times (0,\infty)$ with intensity measure 
\begin{equation}\label{eq:slfvs_intensity_intro}
n dt\otimes n^{\beta} dx\otimes \mu^n(dr).
\end{equation}
Here $n^\beta dx$ denotes the scaling in which the linear dimension of
the infinitesimal region $dx$ is
scaled by $n^\beta$ (so that when we integrate, the volume of a region
is scaled by $n^{\dim\beta}$).
Let 
\begin{equation}
\label{scalings}
u_n = \frac{u}{n^{1-2\beta}}, \qquad\mbox{and}\qquad 
\v{s}_n = 
\frac{1}{\epsilon_n^{2}}\frac{1}{n^{2\beta}}.
\end{equation}
It is convenient to define the constant $\sigma^2$ through
\begin{equation}
\label{defn of sigma}
\sigma^2
=\frac{u}{2\dim}\int_0^{\mc{R}}\int_{\R^\dim}|z|^2\frac{V_r(0,z)}{V_r}dz\mu(dr).
\end{equation}
If $\mu(dr)=\delta_R(r)$, then we recover $\kappa_R$ from~\eqref{def Gamma}.
\begin{theorem} \label{thm:slfvs}
Suppose that $\beta\in(0,1/4)$ and let 
$\epsilon_n$ be a sequence such that $\epsilon_n\to 0$ and 
$(\log n)^{1/2}\epsilon_n\to\infty$ as $n\rightarrow\infty$.
Let $(w_t^n)_{t\geq 0}$ be the {\slfvs} driven by $\Pi^n$ and with
$u_n$, $\v{s}_n$ given by~(\ref{scalings}), and initial condition 
$w_0^n(x)=p(x)$.
Assume that $p$ satisfies
{\Ca}-{\Cc},
and define $\mathscr{T}$, $d(x,t)$ as for Theorem \ref{theorem ac to cf};
take $T^*<\mathscr{T}$.
For $k\in\N$ there exist $n_*(k)<\infty$, 
and $a_*(k),d_*(k)\in(0,\infty)$ such that for all $n\geq n_*$ and all $t$ satisfying 
$a_* \epsilon_n ^2 |\log \epsilon_n |\leq t\leq T^*$, 
\begin{enumerate}
\item for almost every $x$ such that $d(x,\sigma^2 t)\geq d_* \epsilon_n |\log \epsilon_n|$, we 
	have $\E\l[w^n_t(x)\r]\geq 1-\epsilon_n^k$;
\item for almost every $x$ such that $d(x,\sigma^2 t)\leq -d_* \epsilon_n |\log \epsilon_n|$, 
	we have $\E\l[w^n_t(x)\r]\leq \epsilon_n^k$.
\end{enumerate}
\end{theorem}

\begin{remark}
In Section \ref{duality for SLFVS} we explain the origins of these scalings. 
By taking $u_n$ to be small, we are assuming that local population density is high.

By adapting ideas from~\cite{etheridge/freeman/penington/straulino:2015}, we
expect an analogous result for values of $u_n$ up to $\mc O(1)$, but
at the expense of having to take $\epsilon_n\rightarrow 0$ extremely slowly
(so that $\epsilon_n^{-1}=o(\log\log n)$). The stronger the genetic drift, that is the bigger $u_n$,
the larger the value of $n$ required for the diffusive rescaling to 
smooth the allele frequencies under the SLFVS
sufficiently for the behaviour to be close to that of 
the differential equation~\eqref{special allen-cahn equation}.
\end{remark}

The rest of the paper is laid out as follows. In 
Section~\ref{proof of ac to cf}
we establish a duality between equation~(\ref{AC1}) and a branching Brownian
motion which we then use to prove Theorem~\ref{theorem ac to cf}. In 
Section~\ref{proof of slfvs to cf} we establish 
an analogous duality between the
SLFVS and a system of branching and coalescing particles and use it
to establish Theorem~\ref{thm:slfvs}.

\section{Proof of Theorem~\ref{theorem ac to cf}}
\label{proof of ac to cf}

\subsection{A probabilistic dual to Equation~(\ref{special allen-cahn equation})}
\label{subsec:dual_bbm}

Our proof of Theorem~\ref{theorem ac to cf} rests on a duality between 
equation~(\ref{special allen-cahn equation}) and a branching Brownian 
motion in which each individual, independently,
follows a Brownian motion during an
exponentially distributed lifetime (with mean $\epsilon^2$) at the end of
which it splits into {\em three}. Although reminiscent of 
the duality between the Fisher-KPP equation and binary branching 
Brownian motion pioneered by
\cite{skorohod:1964} and \cite{mckean:1975}, here there is a slight twist. 
These papers allow
us to deal with equations of the form
$$\frac{\partial v}{\partial t}=\frac{1}{2}\Delta v +Vf(v),$$
where $V$ is a constant (the branching rate in the branching Brownian motion) 
and $f$ is of the form 
$f(v)=\Phi(v)-v$ where $\Phi(v)$ is the probability generating 
function of a non-negative 
integer-valued random variable (the number of offspring of each
individual in the branching Brownian motion). However, the expression for $f$ 
in~(\ref{special allen-cahn equation}) is not of this form.
Instead we adapt ideas from population genetics 
(notably from \cite{krone/neuhauser:1997, neuhauser/krone:1997}).

First, to maintain compatibility with the PDE 
literature, we shall adopt the convention that
\begin{equation}\label{eq:bm_rate_2}
\textit{all Brownian motions run at rate $2$.}
\end{equation}
That is, at time $1$, Brownian motion has variance $2$. 

In contrast to the McKean-Skorohod setting, our representation of
the solution to~(\ref{AC1}) is not just in terms of the spatial positions of
individuals in the branching Brownian motion at a fixed time, but also depends
on their genealogy. In other words, we have a duality between~(\ref{AC1}) and
the {\em historical process} of the branching Brownian motion.

To write this formally, we require some notation for our ternary 
branching Brownian motion.
We write $\v{W}(t)$ for the historical 
process (which traces out the space-time trees that record the spatial position
of all individuals alive at time $s$ for all $s\in [0,t]$).
This process can be constructed formally as the ternary branching Markov
process in which the position of an `individual' alive at time $s$ is
taken to be the whole Brownian path $(W_u)_{0\leq u\leq s}$ followed by
its ancestors.
To record the genealogy of the process we use Ulam-Harris notation
to label individuals in the branching Brownian motion by elements of
$\mc U =\bigcup_{m=0}^\infty \{1,2,3\}^m$.
For example, $(3,1,2)$ is the particle which is the 2\ts{nd} child of the 1\ts{st} child of the 3\ts{rd} child of the initial ancestor $\emptyset$. 
Let $N(t)\subset \mc U$ denote the set of individuals alive at time $t$.
We shall abuse notation slightly and 
write $(W_i(t))_{i\in N(t)}$ for the 
spatial locations of the individuals alive at time $t$,
and $(W_i(s))_{0\leq s\leq t}$ for the unique path that connects 
leaf $i$ to the root.

We say that $\mc{T}$ is a \textit{time-labelled ternary tree} if $\mc T$
is a finite subtree of $\mc U$ and each 
internal vertex $v$ of the tree is labelled with a time $t_v >0$, where 
$t_v$ is strictly greater than the label of the parent 
vertex of $v$. 
Evidently if we ignore the spatial position of individuals,
each realisation of $\v{W}(t)$ 
traces out a time-labelled ternary tree which records the genealogy and 
associates a time to each branching event. We shall use
$\mathcal{T}(\v{W}(t))$ to denote this 
time-labelled ternary tree.  

For a fixed function $p:\R^\dim \to [0,1]$,
we define a voting procedure on $\mathcal{T}(\v{W}(t))$ as follows.
\begin{enumerate}
\item Each leaf $i$ of $\mathcal{T}(\v{W}(t))$, independently, 
votes $1$ with probability $p(W_i(t))$ and otherwise votes $0$.
\item At each branch point in $\mathcal{T}(\v{W}(t))$, the vote of the parent 
particle $j$ is the 
majority vote of the votes of its three children $(j,1)$, $(j,2)$ and $(j,3)$.
\end{enumerate}
This defines an iterative voting procedure, which runs inwards from 
the leaves of $\mathcal{T}(\v{W}(t))$ to the root $\emptyset$. 
\begin{defn}[$\Vote _p $] \label{vote_defn}
With the voting procedure described above, we define
$ \mathbb{V} _p(\v {W}(t)) $
to be the vote associated to the root $\emptyset$. 
\end{defn}

For $x\in \R^\dim$, we write $\P ^\epsilon _x$ for the probability measure 
under which $(\v{W}(t),t \geq 0)$ has the law of the historical process
of ternary branching Brownian motion in $\R ^\dim$ with branching rate $1/\epsilon ^2$ started 
from a single particle at location $x$ at time $0$.
We write $\E ^\epsilon _x$ for the corresponding expectation.
\begin{theorem}\label{thm:ACdual}
Let $p:\R^{\dim}\to [0,1]$. Then 
\begin{equation}\label{eq:udef}
v^\epsilon (t,x)=\P^\epsilon_x\l[\Vote_p(\v{W}(t))=1\r]
\end{equation}
is a solution to equation~\eqref{special allen-cahn equation}
with initial condition $v^\epsilon (0,x)=p(x)$.
\end{theorem}
\begin{proof}(Sketch)

The proof mirrors that of the representation of solutions of
the Fisher-KPP equation in terms of binary branching Brownian motion,
and so we only sketch it.
As usual the idea is to analyse the expression on the right hand side
of~(\ref{eq:udef}) by partitioning on the behaviour of the branching Brownian
motion in the first $\delta t$ of
time and then to take a limit as $\delta t\downarrow 0$.

Throughout the proof we neglect the superscript $\epsilon$ in $\P^\epsilon _x$, 
$\E^\epsilon _x$ and $v^\epsilon $ and the subscript $p$ in $\Vote _p$.
We write $S$ for the time of the first branching event in the branching Brownian motion
and $W_S$ for 
the position of the ancestor at that time. It is convenient to use $E$ for
expectation when it is with respect to the law of Brownian motion ($W_\cdot$), 
preserving
$\E$ for expectation with respect to that of the historical branching Brownian 
motion ($\v{W}(\cdot)$). Let
$V_1, V_2, V_3$ denote the votes of the three offspring created at time $S$.
By the strong Markov property of the branching Brownian motion, and 
the branching property, we see that the $V_i$ are conditionally 
independent given $(S,W_S)$. Moreover, since conditional on $S\leq\delta t$,
the chance of a second 
branch before time $\delta t$ is ${\mathcal O}(\delta t)$, for $s\leq\delta t$,
$$\E_x[V_1|(S,W_S)=(s,y)]=E_{y}[v(t,W_{\delta t-s})]
+{\mathcal O}(\delta t).$$
From this, if we assume enough regularity of $v(t,x)$
(which follows from that of the heat semigroup), 
\begin{equation}
\label{approx v1}
\E_x[V_1|S\leq\delta t]= v(t,x)+{\mathcal O}(\delta t).
\end{equation}
Still conditioning on $S\leq\delta t$, in order for the vote at the root
to be one, at most one of $V_1, V_2, V_3$ can be zero, and so 
using~(\ref{approx v1}) 
and conditional independence of the $V_i$ given $(S,W_S)$,
$$\P_x\l[\Vote(\v{W}(t+\delta t))=1|S\leq\delta t\r]
=v(t,x)^3+3v(t,x)^2(1-v(t,x))+{\mathcal O}(\delta t).$$
Since if $S>\delta t$ the ancestor of the branching Brownian motion simply
follows a Brownian motion over $[0,\delta t]$, 
partitioning over the behaviour of the branching Brownian motion in 
the first $\delta t$ of time gives
\begin{eqnarray*}
v(t+\delta t,x)&=& 
\P_x\l[\Vote(\v{W}(t+\delta t))=1|S\leq\delta t\r] \P\l[S\leq\delta t\r]\\	
&&+\P_{x}\l[\Vote(\v{W}(t+\delta t))=1\,|\,S>\delta t\r]
(1-\P\l[S\leq\delta t\r])\\
&=&
\P_x\l[\Vote(\v{W}(t+\delta t))=1\,|\,S\leq\delta t\r] \P\l[S\leq\delta t\r]\\	
&&+E_x\l[\P_{W_{\delta t}}\l[\Vote(\v{W}(t))=1\r]\r]
(1-\P\l[S\leq\delta t\r])
.
\end{eqnarray*}
Now $\P[S\leq\delta t]=\epsilon^{-2}\delta t+{\mathcal O}(\delta t^2)$ and so
substituting and rearranging (and once again assuming enough regularity of 
$v(t,x)$) we obtain
\begin{eqnarray*}
\lim_{\delta t\rightarrow 0}\frac{v(t+\delta t,x)-v(t,x)}{\delta t}
&=& \epsilon^{-2}\left(v(t,x)^3+3v(t,x)^2(1-v(t,x))-v(t,x)\right)
\\
&&
+\lim_{\delta t\rightarrow 0}\frac{
E_x\l[\P_{W_{\delta t}}\l[\Vote(\v{W}(t))=1\r]\r]-v(t,x)}{\delta t}
\\
&=& \epsilon^{-2}\left(v(t,x)^3+3v(t,x)^2(1-v(t,x))-v(t,x)\right)
\\
&&
+\lim_{\delta t\rightarrow 0}\frac{
E_x\l[v(t, W_{\delta t})\r]-v(t,x)}{\delta t}
\\
&=&\Delta v(t,x)+\epsilon^{-2}v(t,x)(1-v(t,x))(2v(t,x)-1),
\end{eqnarray*}
as required.
\end{proof}

Armed with this representation, the proof of Theorem~\ref{theorem ac to cf} is 
reduced to proving the following result about our branching Brownian motions.
\begin{theorem} \label{thm:BBMtwo}
Suppose $p:\R^\dim \to [0,1]$ is such that {\Ca}-{\Cc} hold. 
Define $\mathscr{T}$, $d(x,t)$ as for Theorem~\ref{theorem ac to cf};
fix $T^*\in (0,\mathscr{T})$ and let $k\in\N$. 
There exist $\epsilon_{\dim}(k)>0$, and $a_{\dim}(k),c_{\dim}(k)\in(0,\infty)$ such that for all $\epsilon\in(0,\epsilon_{\dim})$ and $t$ satisfying $a_{\dim}\epsilon ^2 |\log \epsilon |\leq t\leq T^*,$
\begin{enumerate}
\item for $x$ such that $d(x,t)\geq c_{\dim}\epsilon |\log \epsilon|$, we have $\P^\epsilon_x\l[\Vote _p(\v{W}(t))=1\r]\geq 1-\epsilon^k$;
\item for $x$ such that $d(x,t)\leq -c_{\dim}\epsilon |\log \epsilon|$, we have $\P^\epsilon_x\l[\Vote _p (\v{W}(t))=1\r]\leq \epsilon^k$.
\end{enumerate}
\end{theorem}

The proof of Theorem~\ref{thm:BBMtwo} will proceed in two steps. 
First, in Section~\ref{sec:BBMone}, 
we prove a one-dimensional analogue of the result in 
the special case in which $p(x)=\1\{x\geq 0\}$. The proof rests on symmetry
of branching Brownian motion and the monotonicity that results from
the specific choice of initial condition $p$.
The second step uses the definition of mean curvature flow and the regularity properties
that follow from the conditions {\Ca}-{\Cc}. These allow us to couple 
the distance between the (backwards in time) mean curvature flow 
$(\v {\Gamma}_{t-s})_{s \in [0,t]}$ and a 
(forwards in time) $\dim$-dimensional Brownian motion $W$ with a 
(forwards in time) one-dimensional Brownian motion $B$ in such a way that
$d(W_s,t-s)$ is well approximated by $B_s$
when $W_s$ is close to $\v {\Gamma}_{t-s}$.
This coupling is made precise in Proposition~\ref{prop:coupling1} in Section~\ref{sec:CFsec}. 
The proof of Theorem~\ref{thm:BBMtwo}, which combines these two steps by
bounding the errors that occur far from the interface $\v {\Gamma} _{t-s}$, 
can be found in Section~\ref{sec:BBMtwo}. 

\begin{notation}
It is convenient to have a prominent distinction between one dimensional
and multi-dimensional Brownian motion in our notation. We therefore adopt the
convention that $B$ will denote one dimensional Brownian motion and
$\v{B}$ will represent the corresponding historical branching Brownian motion
and we preserve $W$ and $\v{W}$ for dimensions $\dim\geq 2$.
\end{notation}

\subsection{Majority voting in one dimensional BBM}
\label{sec:BBMone}

In this section we consider only ternary branching Brownian motion 
in dimension $\dim=1$. 

As in Section~\ref{subsec:dual_bbm}, 
for $x\in \R$, we write $\P ^\epsilon _x$ for the probability measure 
under which $(\v{B}(t),t \geq 0)$ has the law of historical ternary branching
Brownian motion in $\R$ with branching rate $1/\epsilon ^2$ started from a 
single particle at location $x$ at time $0$,
and $\E ^\epsilon _x$ for the corresponding expectation.
We also write $P_x$ for the probability measure under which 
$(B_t)_{t \geq 0}$ has the law of a Brownian motion started at $x$,
and $E_x$ for the corresponding expectation.

Throughout this section we write $\mathbb V := \mathbb V_{p_0}$ where
$p_0(x)=\1\{x\geq 0\}$, so that
a leaf votes $1$ if and only if it is in the right half line.
Our aim is to prove the following one-dimensional analogue of 
Theorem~\ref{thm:BBMtwo} for this initial condition $p_0$.

\begin{theorem} \label{thm:BBMone}
Let $T^*\in(0,\infty)$. For all $k\in\N$ there exist $c_1(k)$ and $\epsilon_1(k)>0$ such that, for all $t\in[0,T^*]$ and all $\epsilon\in(0,\epsilon_1)$,
\begin{enumerate}
\item for $z\geq c_1(k)\epsilon |\log \epsilon |$, we have $\P^\epsilon_z \l[ \Vote (\v{B}(t))=1\r]\geq 1-\epsilon ^k$ 
\item for $z\leq -c_1(k)\epsilon |\log \epsilon |$, we have $\P^\epsilon_z \l[ \Vote (\v{B}(t))=1\r]\leq \epsilon ^k.$ 
\end{enumerate}
\end{theorem}

\begin{remark}
The subscript $1$ on $a_1,c_1$ and $\epsilon_1$ is to emphasize that
Theorem~\ref{thm:BBMone} applies in dimension $1$. We shall 
often suppress the dependence on $k$ in our notation.
\end{remark}

Note that, if $z\geq 0$, then a typical leaf of the branching Brownian motion is more 
likely to vote $1$ than $0$, and that the opposite is true for $z<0$. 
Theorem~\ref{thm:BBMone} says that the majority 
voting procedure magnifies a small voting bias at the leaves 
into a much stronger voting bias at the root. If the votes of different leaves
were independent this would be elementary, but the spatial
structure of the branching Brownian motion introduces strong correlations
between votes of closely related individuals. To overcome this, we first 
use a symmetry argument to show that the bias
close to the root will be at least as strong as that at the leaves 
and then check that, 
as $\epsilon$ tends to zero, there
is enough branching close to the root to sufficiently magnify the bias.

\subsubsection{Proof of Theorem~\ref{thm:BBMone}}

First note that
with our
special choice of initial condition $p_0$, for any 
$x_1 \leq x_2\in\R$, 
\begin{equation}\label{eq:monotonicity}
\P^\epsilon_{x_1}[\Vote(\v{B}(t))=1]\leq \P^\epsilon_{x_2}[\Vote(\v{B}(t))=1].
\end{equation}

By analogy with the previous subsection, we use
$\mathcal{T}(\v{B}(t))$ to denote the  
time-labelled tree traced out by the branching Brownian motion
up to time $t$, 
and for any time-labelled ternary tree $\mc{T}$ we write
\begin{equation}\label{PTdef}
\P^{t}_x (\mc{T})=\P^\epsilon _x\l[\Vote (\v{B}(t))=1
\|\mathcal{T}(\v{B}(t))=\mc{T}\r].
\end{equation}
By the symmetry of the Brownian motions followed by individuals in 
$\v{B}(t)$ conditional on $\{\mathcal{T}(\v{B}(t))=\mc{T}\}$, 
applying the reflection $x\mapsto -x$ to the process, we see that 
for any time-labelled ternary tree $\mc{T}$, any time $t>0$, and any $z\in \R$,
\begin{equation}
	\label{prob_symmetry}
	\P_{z}^t(\mc{T})=1-\P_{-z}^t(\mc{T}).
\end{equation}
The monotonicity in \eqref{eq:monotonicity} and the symmetry in \eqref{prob_symmetry} are key to our proof of Theorem \ref{thm:BBMone}.

Taking $z=0$ in~(\ref{prob_symmetry}) shows that $\P^t_0(\mc{T})=\frac{1}{2}$ 
for all $t>0$, and, by \eqref{eq:monotonicity}, for all $t>0$ and all time-labelled 
ternary trees $\mc T$ we have
$$
\P^t_z(\mc{T})\geq\tfrac{1}{2}\;\text{ for } z>0;\qquad 
\P^t_z(\mc{T})\leq\tfrac{1}{2}\;\text{ for } z<0. 
$$

We now introduce notation for the majority voting procedure. 
Let $g:[0,1]^3\to[0,1]$ be given by
\begin{equation} \label{g_defn}
g(p_1,p_2,p_3)=p_1 p_2 p_3 +p_1 p_2 (1-p_3)+p_2 p_3(1-p_1)+p_3 p_1 (1-p_2). 
\end{equation}
This is the probability that a majority vote gives the result $1$, 
in the special case where the three voters are independent and have 
probabilities $p_1$, $p_2$ and $p_3$ respectively of voting $1$. 
With a slight abuse of notation, we let $g(p)=g(p,p,p)$, for $p\in [0,1]$. 
Note that
\begin{equation}\label{gantisym}
g(1-p_1,1-p_2,1-p_3)=1-g(p_1,p_2,p_3).
\end{equation}

For $\mc T$ a time-labelled ternary tree with at least one branching event,
suppose that the time to the 
first branching event in $\mc T$ is $\tau$ and that the subtrees with time 
labels corresponding to the (descendants of the) three offspring from the 
branching event are $\mc T_1$, $\mc T_2$ and $\mc T_3$ (here a vertex 
$v$ with time label $t_v$ in $\mc T$ is given time label $t_v-\tau$ 
in $\mc T_i$).
Then, we write
\begin{equation}\label{eq:star_notation}
g\l(\P^{t-\tau}_{B_{\tau}}(\mc{T}\star)\r)=g\l(\P^{t-\tau}_{B_{\tau}}(\mc{T}_1), 
\P^{t-\tau}_{B_{\tau}}(\mc{T}_2),\P^{t-\tau}_{B_{\tau}}(\mc{T}_3)\r)
\end{equation}
and the identity
\begin{equation}\label{branchid}
\P^t_{z}(\mc{T})=E_z \l[g\l(\P^{t-\tau}_{B_{\tau}}(\mc{T}\star)\r)\r]
\end{equation}
expresses the majority voting that takes place at the 
first branch of $\mc{T}$.

Our next lemma states that the majority voting procedure cannot reduce the voting bias. In view of symmetry \eqref{prob_symmetry}, when it is convenient to do so we will only state such results for the case $z\geq 0$.

\begin{lemma} \label{no_demagnification}
For any time-labelled ternary tree $\mc T$, any time $t>0$, and any $z\geq 0$,
$$\P_z^t(\mc T)\geq P_z[B_t\geq 0].$$
\end{lemma}
\begin{proof}
The proof is by induction on the number of branching events in the tree $\mc{T}$. Let $\mc{T}_0$ denote the tree with a root and a single leaf. Then, by definition, 
$
\P_z^t(\mc{T}_0)=P_z\l[B_t\geq 0\r].
$

We now approach the inductive step. Suppose that the statement of the lemma holds for all time-labelled ternary trees with up to $n$ internal vertices. We define $h:[0,1]^3\to\R$ by
$$h(p_1,p_2,p_3)=g(p_1,p_2,p_3)-\frac{1}{3}(p_1+p_2+p_3),$$
and note that from~\eqref{gantisym} we have
\begin{equation}\label{hantisym}
h(1-p_1,1-p_2,1-p_3)=-h(p_1,p_2,p_3).
\end{equation}

We can write $h$ in the form
\begin{equation*}
h(p_1,p_2,p_3)=\tfrac{1}{3}\sum p_{i_1}\Big((1-p_{i_2})(p_{i_3}-\tfrac{1}{2})+(1-p_{i_3})(p_{i_2}-\tfrac{1}{2}) \Big) 
\end{equation*}
where the sum is over $(i_1,i_2,i_3)=(1,2,3),(2,3,1),(3,1,2)$. Hence 
\begin{equation} \label{hpos}
\frac{1}{2}\leq p_1,p_2,p_3\leq 1\;\ra\;h(p_1,p_2,p_3)\geq 0.
\end{equation}
We will use the $\star$ notation defined in~\eqref{eq:star_notation} for $h$ in the same way as we use it for $g$. 

Suppose that $\mc{T}$ is a time-labelled ternary tree with $n+1$ internal 
vertices and let $\tau$, $\mc{T}_1$, $\mc{T}_2$, $\mc{T}_3$ be as 
in~\eqref{branchid}. Using~\eqref{branchid}, by the definition of $g$ and 
$h$ we have
\begin{align}
\P^t_{z}(\mc{T})
&=E_z\l[g\l(\P^{t-\tau}_{B_{\tau}}(\mc{T}\star)\r)\r] \notag \\
&=E_z\l[h\l(\P^{t-\tau}_{B_{\tau}}(\mc{T}\star)\r)\r]
+\frac{1}{3}\sum\limits_{i=1}^3 E_z\l[\P^{t-\tau}_{B_{\tau}}(\mc{T}_i)\r]. 
\label{hbreakdown}
\end{align}
We will show that the first term of \eqref{hbreakdown} is non-negative. 
Combining~(\ref{hantisym}) with~(\ref{prob_symmetry}),
$$h(\P^{t-\tau}_{B_{\tau}}(\mc{T}\star))
=-h(\P^{t-\tau}_{-B_{\tau}}(\mc{T}\star)).
$$
Hence,
\begin{align}
E_z[h(&\P^{t-\tau}_{B_{\tau}}(\mc T\star))]\notag\\
&=E_z \left[h(\P^{t-\tau}_{B_{\tau}}(\mc T\star))
\1\left\{B_{\tau}\geq 0\right\}\right]
+E_z\left[h(\P^{t-\tau}_{B_{\tau}}(\mc T\star))
\1\left\{B_{\tau}<0\right\}\right]\notag\\
&=E_z\left[h(\P^{t-\tau}_{B_{\tau}}(\mc T\star))
\1\left\{B_{\tau}\geq 0\right\}\right]
-E_z\left[h(\P^{t-\tau}_{-B_{\tau}}(\mc T\star))
\1\left\{B_{\tau}<0\right\}\right]\notag\\
&=\int_{0}^\infty h(\P^{t-\tau}_{x}(\mc T\star))(\phi_{z,2\tau}(x)
-\phi_{z,2\tau}(-x))\,dx,\label{eq:fliptrick}
\end{align}
where $\phi_{\mu,\sigma^2}$ denotes the density of a $N(\mu,\sigma^2)$ 
random variable. Since $\P^{t-\tau}_{x}(\mc{T}_i)\geq 1/2$ for $x\geq 0$,
by \eqref{hpos} we have $h(\P^{t-\delta t}_{x}(\mc{T}\star))\geq 0$, and since $z\geq 0$, 
for all $x\geq 0$ we have
$$\phi_{z,2\tau}(x)-\phi_{z,2\tau}(-x)\geq 0,$$
which proves that~\eqref{eq:fliptrick} is non-negative. This shows that the first term of \eqref{hbreakdown} is non-negative and we now move on to the second term.

Using our inductive hypothesis, for $i=1,2,3$,
$$E_z[\P^{t-\tau}_{B_{\tau}}(\mc{T}_i)]\geq E_z\l[P_{B_\tau}\l[B_{t-\tau}\geq0\r]\r]
=P_z[B_t\geq 0]$$
and so substituting into~\eqref{hbreakdown}
completes the proof of
Lemma~\ref{no_demagnification}.
\end{proof}

Our next task is to show that successive rounds of majority voting magnify 
a small bias at the leaves into a large bias at the root of a tree.
Recall that for $p\in [0,1]$,
$$g(p):=g(p,p,p)=3p^2-2p^3,$$ 
and define $g^{(n)}(p)$, inductively, by
\begin{equation*}
g^{(1)}(p)=g(p), \qquad
g^{(n+1)}(p)=g^{(n)}(g(p)).
\end{equation*}
Thus, $g^{(n)}(p)$ describes the probability of voting $1$ at the root 
of an $n$-level regular ternary tree if the votes of the leaves are 
i.i.d.~Bernoulli$(p)$.
\begin{lemma} \label{g_iteration}
For all $k\in\N$ there exists $A(k)<\infty$ such that, for all $\epsilon \in (0,\frac12]$ and $n\geq A(k)|\log \epsilon |$ we have
$$g^{(n)}(\tfrac{1}{2}+\epsilon)\geq 1-\epsilon^k.$$
\end{lemma}
\begin{proof}
We carry out two phases of iteration of $g$. 
First, we will show that it takes $\mc{O}(|\log \epsilon|)$ iterations 
to obtain 
\begin{equation}\label{eq:g_iter_1}
g^{(n)}(\tfrac{1}{2}+\epsilon)\geq \tfrac{1}{2}+\tfrac{1}{\sqrt{8}}.
\end{equation}
Then we note that $\mc{O}(\log |k \log \epsilon|)$ iterations are 
required to obtain 
\begin{equation}\label{eq:g_iter_2}
g^{(n)}(\tfrac{1}{2}+\tfrac{1}{\sqrt{8}})\geq 1-\epsilon^k. 
\end{equation}
Since $g$ is monotone, combining the two phases completes the proof.

For the first phase, if $\delta\in(0,1/\sqrt{8})$ then a simple calculation 
shows that 
$$g(\tfrac{1}{2}+\delta)=\tfrac{1}{2}+
\tfrac{3}{2}\delta-2\delta^3\geq\tfrac{1}{2}+\tfrac{5}{4}\delta.$$ 
Thus if $g^{(n)}(\frac{1}{2}+\epsilon)-\frac{1}{2}<1/\sqrt{8}$, we have
\begin{align*}
	g^{(n+1)}(\tfrac{1}{2}+\epsilon)-\tfrac{1}{2}
	\geq \tfrac{5}{4}\left(g^{(n)}(\tfrac{1}{2}+\epsilon)-\tfrac{1}{2}\right)
\geq (\tfrac{5}{4})^{n}\epsilon.
\end{align*} 
It follows immediately that $\mc{O}(|\log \epsilon|)$ iterations are required to achieve~\eqref{eq:g_iter_1}.

For the second phase, note that 
$1-g(1-\delta)=3\delta^2-2\delta^3\leq 3\delta^2$, so that
$$1-g^{(n+1)}(\tfrac{1}{2}+\tfrac{1}{\sqrt{8}})\leq 3 
\left(1-g^{(n)}(\tfrac{1}{2}+\tfrac{1}{\sqrt{8}})\right)^2\leq \tfrac{1}{3}\Big(3(\tfrac{1}{2}-\tfrac{1}{\sqrt{8}})\Big)^{2^n}.$$
Noting that $3(\frac{1}{2}-\frac{1}{\sqrt{8}})<1$, it follows easily that the number of iterations required to obtain~\eqref{eq:g_iter_2} is $\mc{O}(\log|k\log\epsilon|)$.
\end{proof}

We now want to see that there is a (large)
regular ternary tree sitting inside $\mc{T}(\v{B}(t))$.
Let $\mc{T}^{reg}_n=\cup_{k\leq n}\{1,2,3\}^k \subset \mc U$ denote the $n$-level regular ternary tree and, for $l\in\R$, let $\mc{T}^{reg}_l=\mc{T}^{reg}_{\lceil l\rceil}$.
For $\mc T$ a time-labelled ternary tree, we use the relation $\mc T\supseteq \mc{T}^{reg}_l$ to mean that as subtrees of $\mc U$, $\mc{T}^{reg}_l$ is contained inside $\mc T$ (ignoring its time labels).

\begin{lemma} \label{ternary_tree}
Let $k\in\N$ and let $A=A(k)$ be as in Lemma~\ref{g_iteration}. Then 
there exist $a_1=a_1(k)$ and $\epsilon_1=\epsilon_1(k)$ such that, for 
all $\epsilon\in(0,\epsilon_1)$ and $t\geq a_1\epsilon^2|\log \epsilon|$, 
$$\P^\epsilon \l[\mathcal{T}(\v{B}(t))\supseteq \mc{T}^{reg}_{A(k)|\log \epsilon |}\r]\geq 1- \epsilon ^k .$$
\end{lemma}
\begin{proof}
First we establish control over the tail distribution of the sum of
$n$ independent exponentially distributed (branching) times.
Suppose $(X_j)_{j\geq 1}$ are i.i.d.~Exp(1) random variables and let $S_n=\sum_{j=1}^n X_j$. Then 
$$ M_X(\lambda )=\E \l[ e^{\lambda X} \r]= 
\begin{cases}
   \frac{1}{1-\lambda} & \text{if } \lambda <1 \\
   \infty       & \text{if } \lambda \geq 1
  \end{cases}$$
and for $a\geq 1$,
$$\Psi^*(a):=\sup_{\lambda \geq 0}(\lambda a -\log M_X(\lambda))=\sup_{0\leq \lambda <1}(\lambda a+\log (1-\lambda))=a-1-\log a. $$
By Cram\'{e}r's theorem, for $a\geq 1$,
\begin{equation} \label{cramer}
\lim_{n\rightarrow \infty} \l( -\frac{1}{n}\log \P [S_n \geq na] \r)=\Psi^* (a)=a-1-\log a. 
\end{equation}

Suppose $a\geq 1$. 
For each leaf of $\mc{T}^{reg}_l$ we use~\eqref{cramer} to 
estimate the probability that it is not
in $\mc{T}(\v{B}(t))$ and combine with a union bound (summing over leaves).
For $t\geq a \epsilon ^2 \lceil A|\log \epsilon | \rceil$ we have
\begin{align}
&\P^\epsilon \l[\mathcal{T}(\v{B}(t))\nsupseteq \mc T^{reg}_{A|\log \epsilon|}\r]
\notag\\
&\hspace{1cm}\leq 3^{\lceil A|\log \epsilon |\rceil}\P \l[\epsilon ^2 S_{\lceil A|\log \epsilon |\rceil}\geq a\epsilon ^2 \lceil A|\log \epsilon|\rceil\r]\notag\\
&\hspace{1cm}=\exp\l(\lceil A|\log \epsilon|\rceil \l(\log 3 + \frac{1}{\lceil A|\log \epsilon|\rceil }\log \P\l[S_{\lceil A|\log \epsilon |\rceil}\geq a\lceil A|\log \epsilon |\rceil\r]\r)\r).\label{ugh}
\end{align}
By~\eqref{cramer} (with $n=\lceil A|\log\epsilon|\rceil$), 
we can choose $\epsilon_1(k)<e^{-1}$ such that, for all $\epsilon\in(0,\epsilon_1)$, 
$$\frac{1}{\lceil A|\log \epsilon|\rceil }\log \P\l[S_{\lceil A|\log \epsilon |\rceil}\geq a\lceil A|\log \epsilon |\rceil\r]\leq -a+3/2+\log a.$$
Choose $a\geq 1$ sufficiently large that $-a+3/2+\log a\leq -\log 3 -k/A$.
Putting this into~\eqref{ugh} we obtain
$$
\P^\epsilon \l[\mathcal{T}(\v{B}(t))\nsupseteq \mc T^{reg}_{A|\log \epsilon|}\r]
\leq \exp\l(-|\log\epsilon| k\r)
$$
for $t\geq a \epsilon ^2 \lceil A|\log \epsilon | \rceil$. Letting $a_1=a(A+1)$ completes the proof.
\end{proof}

We now control the maximal displacement of individuals in the ternary
branching Brownian motion at small times.
Let $N(t)$ denote the set of individuals alive in $\v B(t)$.
\begin{lemma} \label{all_particles_bound}
Let $k\in\N$, and let $a_1(k)$ be as in Lemma~\ref{ternary_tree}. Then there exist $d_1(k)$, $\epsilon_1(k)$ such that, for all $\epsilon \in (0,\epsilon_1(k))$ and all $s\leq a_1\epsilon^2|\log\epsilon|$, 
$$\P^\epsilon _x\l[\exists i\in N(s):|B_i(s)-x|\geq d_1(k)\epsilon |\log \epsilon | \r] \leq \epsilon ^k.$$
\end{lemma}
\begin{proof}
Write $\delta_1 =a_1\epsilon^2|\log\epsilon|$ and let $Z$ be a $N(0,1)$
distributed random variable. 
By Markov's inequality, for $s\leq\delta_1$ we have
\begin{align*}
\P^\epsilon_x\l[\exists i\in N(s):|B_i(s)-x|\geq d_1 \epsilon |\log \epsilon | \r]
&\leq \E^\epsilon\l[|N(s)|\r] \P\l[\sqrt{2s}|Z|\geq d_1\epsilon |\log \epsilon |\r]\\
&\leq \E^\epsilon \l[|N(\delta_1)|\r] \P\l[\sqrt{2\delta_1}|Z|\geq d_1\epsilon |\log \epsilon |\r]\\
&=e^{2 \delta_1 /\epsilon^2} \P \l[\sqrt{2a_1}|Z|\geq d_1 |\log \epsilon |^{1/2}\r]\\
&\leq \frac{1}{\epsilon ^{2a_1}}\exp\l(-\tfrac{1}{4}\tfrac{d_1^2}{a_1} |\log \epsilon |\r)\\
&=\epsilon ^{\tfrac{1}{4} \tfrac{d_1^2}{a_1} -2a_1}.
\end{align*}
Here the fourth line holds for $\epsilon >0$ sufficiently small.
The proof is completed by choosing $d_1=d_1(k)$ large enough that $\frac{d_1^2}{4a_1}-2a_1\geq k$. 
\end{proof}

We now have all the ingredients needed to prove Theorem~\ref{thm:BBMone}. 
If $z \geq 2d_1 \epsilon |\log \epsilon |$, then, at time 
$\delta_1 =a_1\epsilon^2|\log \epsilon |$, by Lemma~\ref{all_particles_bound},
with high probability, all individuals in $\v{B}(\delta_1)$
are still $\geq d_1 \epsilon |\log\epsilon|$. Lemma~\ref{no_demagnification}
tells us that there is a positive voting bias at each of those points and
Lemma~\ref{ternary_tree} shows that this will be magnified by at least
$\mc O (|\log\epsilon|)$ rounds of majority voting as we trace back to the 
root. Finally, Lemma~\ref{g_iteration} gives us a lower bound on the 
bias at the root. 

\begin{proof} [Of Theorem~\ref{thm:BBMone}.] We will prove the first 
statement of the theorem; the second then follows by symmetry.

For all $\epsilon<1/2$, define $z_\epsilon $ implicitly by the relation 
$\P\l[B_{T^*}\geq -z_\epsilon \r]=\tfrac{1}{2}+\epsilon$, and note that 
$z_\epsilon \sim \epsilon \sqrt{4 \pi T^*}$ as $\epsilon \rightarrow 0$. 
Let $\epsilon_1(k)<1/2$ be sufficiently small that Lemmas~\ref{ternary_tree} 
and~\ref{all_particles_bound} hold for $\epsilon \in (0, \epsilon_1(k))$.
Let $d_1(k)$ be given by Lemma~\ref{all_particles_bound} and let 
$c_1(k)=2d_1(k)$ so that (by reducing $\epsilon_1$ if necessary), 
for $\epsilon\in (0,\epsilon_1)$, 
\begin{equation}\label{eq:room for shift}
d_1(k)\epsilon|\log \epsilon |+z_\epsilon\leq c_1(k)\epsilon|\log \epsilon|.
\end{equation}
Let $a_1(k)$ be given by Lemma~\ref{ternary_tree} and let 
\begin{equation}
\label{delta 1}
\delta_1=\delta_1(k,\epsilon)=a_1(k)\epsilon^2|\log\epsilon|.
\end{equation}

If $t\in(0,\delta_1)$ and $z\geq c_1\epsilon|\log \epsilon |$, then
\begin{align*}
\P^\epsilon_z\l[\Vote(\v{B}(t))=0\r]
&\leq \P^\epsilon_z\big[\exists i\in N(t)\text{ such that }|B_i(t)-z|\geq d_1\epsilon |\log \epsilon | \big]\\ 
&\leq \epsilon ^k, 
\end{align*}
where the second line follows by Lemma~\ref{all_particles_bound}.

We now suppose that $t\in [\delta_1,T^*]$ and $z\geq c_1 \epsilon |\log \epsilon |$.
Let $\mathcal T_{\delta_1}=\mathcal T(\v{B}(\delta_1))$ 
denote the time-labelled tree of the branching Brownian motion
up to time $\delta_1$. We define 
$$p_{t-\delta_1}(z)=\P_z^\epsilon\l[\Vote(\v{B}(t-\delta_1))=1\r],$$
and 
$$p_{t-\delta_1}^\epsilon(z) = p_{t-\delta_1}(z_\epsilon), 
\quad\mbox{ for all }z\in\R.$$ 
Finally, write $\{\v{B}(\delta_1)>z_\epsilon\}$ for the event
$B_i(\delta_1)>z_\epsilon$ for all $i\in N(\delta_1)$. 
Then, 
\begin{align}
\P^\epsilon_z\l[\Vote(\v{B}(t))=1\r]&=
	\P^\epsilon_z\l[\Vote_{p_{t-\delta_1}(z)}(\v{B}(\delta_1))=1\r]
	\notag\\
	&\geq	
	\P^\epsilon_z\l[\l\{\Vote_{p^\epsilon_{t-\delta_1}(z)}(\v{B}(\delta_1))=1
	\r\}\cap\l\{\v{B}(\delta_1)>z_\epsilon\r\}\r]\notag
\\
	&\geq	
	\P^\epsilon_z\l[\Vote_{p^\epsilon_{t-\delta_1}(z)}(\v{B}(\delta_1))=1\r]
	-\epsilon^k.\label{captureandbound}
\end{align}
Here, the first line follows by the Markov property of $\v{B}$ at time $\delta_1$. 
The second follows by the monotonicity property \eqref{eq:monotonicity}. The third line then follows by Lemma \ref{all_particles_bound}, using \eqref{eq:room for shift} and our hypothesis that $z \geq c_1\epsilon|\log\epsilon|$.

We have 
\begin{equation}\label{l13appl}
p^\epsilon_{t-\delta_1}(z)\geq
P_{z_\epsilon}\l[B_{t-\delta_1}\geq 0\r]\geq \tfrac{1}{2}+\epsilon. 
\end{equation}
Here, the first inequality follows from Lemma~\ref{no_demagnification}. The second follows by the definition of $z_\epsilon$, since $t-\delta_1<T^*$. 

If $p_i\geq 1/2$ for $i=1,2,3$ then \eqref{hpos} implies that $g(p_1,p_2,p_3)\geq \min(p_1,p_2,p_3)$. Hence, if
each leaf of $\mc{T}_{\delta_1}$ votes $1$ independently with 
probability at least $\tfrac{1}{2}+\epsilon$ and 
$\mathcal T_{\delta_1}\supseteq \mc T^{reg}_{A|\log \epsilon|}$, 
then 
each of the leaves of $\mc T^{reg}_{A|\log \epsilon|}$ votes 1 
independently with probability at least $\tfrac{1}{2}+\epsilon$. Therefore,
$$\P^\epsilon_z\l[\Vote(\v{B}(t))=1\r]\geq 
g^{(\lceil A|\log \epsilon| \rceil)}(\tfrac{1}{2}+\epsilon)-2\epsilon^k
\geq 1-3\epsilon ^k.$$
Here, the first inequality follows by substituting \eqref{l13appl} into \eqref{captureandbound} and then applying Lemma~\ref{ternary_tree} and the second then follows by Lemma~\ref{g_iteration}. This completes the proof. 
\end{proof}

\subsubsection{The slope of the interface}

In proving Theorem~\ref{thm:BBMtwo} we shall also exploit a 
lower bound on the `slope' of the interface in $\dim =1$ which we prove in this
subsection. We obtain it as a corollary of the following result.

\begin{prop} \label{prop:deriv1}
Suppose $x\geq 0$ and $\eta>0$. Then for any time-labelled ternary tree $\mc{T}$ and any time $t$,
$$\P_{x}^t(\mc{T})-\P_{x-\eta}^t(\mc{T})\geq \P_{x+\eta}^t(\mc{T})-\P_{x}^t(\mc{T}). $$
\end{prop}

\begin{proof}
The proof is by induction on the number of branching events in $\mc T$, and is similar to the proof of Lemma~\ref{no_demagnification}.
For $\mc{T}_0$ a (time-labelled) tree with a root and a single leaf, we have
$$ \P_{x}^t(\mc{T}_0)-\P_{x-\eta}^t(\mc{T}_0)=\int_{x-\eta}^{x}\phi_{0,2t}(u)\,du\geq \int_x^{x+\eta}\phi_{0,2t}(u)\,du=\P_{x+\eta}^t(\mc{T}_0)-\P_{x}^t(\mc{T}_0)$$
where $\phi_{\mu,\sigma^2}$ is the density of a N$(\mu,\sigma ^2)$ random variable. 

Now, assume that the lemma holds for all time-labelled ternary trees with 
at most $n$ internal vertices. Let $\mc{T}$ be a time-labelled ternary tree 
with $n+1$ internal vertices and suppose that the time to the first branching 
event of $\mc{T}$ is $\tau$ and let $\mc{T}_1$, $\mc{T}_2$, $\mc{T}_3$ 
denote the trees of the three offspring of that branching. Then using 
the notation of~\eqref{eq:star_notation},
\begin{align}
&\l( \P_{x}^t(\mc{T})-\P_{x-\eta}^t(\mc{T})\r)-\l(\P_{x+\eta}^t(\mc{T})-\P_{x}^t(\mc{T})\r)\notag\\
&\hspace{0.5pc}=\l(E_{x}\l[g(\P_{B_{\tau}}^{t-\tau}(\mc{T}\star))\r]-
	E_{x-\eta}\l[g(\P_{B_{\tau}}^{t-\tau}(\mc{T}\star))\r]\r)
	-\l(E_{x+\eta}\l[g(\P_{B_{\tau}}^{t-\tau}(\mc{T}\star))\r]-
	E_{x}\l[g(\P_{B_{\tau}}^{t-\tau}(\mc{T}\star))\r]\r)\notag\\
&\hspace{0.5pc}=\int_{-\infty}^\infty\l\{ \big(g(\P^{t-\tau}_{y}(\mc{T}\star))-
	g(\P^{t-\tau}_{y-\eta}(\mc{T}\star))\big)-
\big(g(\P^{t-\tau}_{y+\eta}(\mc{T}\star))-
g(\P^{t-\tau}_{y}(\mc{T}\star))\big)\r\}\phi_{x,2\tau}(y)dy\notag\\
&\hspace{0.5pc}=\int_0^\infty \l\{\big(g(\P^{t-\tau}_{y}(\mc{T}\star))-
g(\P^{t-\tau}_{y-\eta}(\mc{T}\star))\big)-
\big(g(\P^{t-\tau}_{y+\eta}(\mc{T}\star))-g(\P^{t-\tau}_{y}(\mc{T}\star))
\big)\r\}(\phi_{x,2\tau}(y)-\phi_{x,2\tau}(-y))\,dy.\label{eq:integ}
\end{align}
Here, the second line follows by~\eqref{branchid} and the last line follows 
from~\eqref{gantisym} and~(\ref{prob_symmetry}), which imply that 
$g(\P^t_w(\mc{T}\star))=1-g(\P^t_{-w}(\mc{T}\star))$. Note the similarity 
to~\eqref{eq:fliptrick}.

Since $x\geq 0$, we have
\begin{equation}\label{eq:phipve}
\phi_{x,2\tau}(y)-\phi_{x,2\tau}(-y)\geq 0
\end{equation}
for $y\geq 0$. 
In view of \eqref{eq:integ} we should like to check that for $y\geq 0$
\begin{equation}\label{eq:integpve}
\l(g(\P^{t-\tau}_{y}(\mc{T}\star ))-g(\P^{t-\tau}_{y-\eta}(\mc{T}\star ))\r)
-\l(g(\P^{t-\tau}_{y+\eta}(\mc{T}\star ))-g(\P^{t-\tau}_{y}(\mc{T}\star ))\r)\geq 0.
\end{equation}
By our inductive hypothesis, for $y\geq 0$ we have
\begin{equation*}
\big(\P^{t-\tau}_{y}(\mc{T}_i)-\P^{t-\tau}_{y-\eta}(\mc{T}_i)\big)-
\big(\P^{t-\tau}_{y+\eta}(\mc{T}_i)-\P^{t-\tau}_{y}(\mc{T}_i)\big)\geq 0, 
\end{equation*}
and so by monotonicity of $g$, for~\eqref{eq:integpve}
it is enough to check that 
\begin{equation}
	\label{symmetric version}
g(\P^{t-\tau}_{y+\eta}(\mc{T}\star ))-2g(\P^{t-\tau}_{y}(\mc{T}\star ))
+g\l(\P^{t-\tau}_{y}(\mc{T}\star )-(\P^{t-\tau}_{y+\eta}(\mc{T}\star )
-\P^{t-\tau}_{y}(\mc{T}\star ))\r)\leq 0.
\end{equation}
To see that \eqref{symmetric version} holds, 
note that 
\begin{align}
&g(p_1+\eta_1,p_2+\eta_{2},p_3+\eta_3)-2g(p_1,p_2,p_3)+g(p_1-\eta_1,p_2-\eta_{2},p_3-\eta_3)\notag\\
&\hspace{2pc}
=2\eta_1 \eta_2 (1-2p_3)+2\eta_2 \eta_3 (1-2p_1)+2\eta_3 \eta_1 (1-2p_2).
\notag 
\end{align}
and set $p_i =\P^{t-\tau}_{y}(\mc{T}_i) $ and 
$\eta_i=\P^{t-\tau}_{y+\eta}(\mc{T}_i)-\P^{t-\tau}_{y}(\mc{T}_i)$. 
Since for $y\geq 0$, $p_i \geq 1/2$, the inequality \eqref{symmetric version} 
then follows.

Putting \eqref{eq:phipve} and \eqref{eq:integpve} into~\eqref{eq:integ} 
completes the inductive step, which in turn completes the proof.
\end{proof}

\begin{cor} \label{lem:high_deriv} 
Take $\epsilon_1(1)$ and $c_1(1)$ from Theorem~\ref{thm:BBMone}.
Let $\epsilon <\min(\epsilon_1(1),\tfrac{1}{24})$.
Suppose that for some $t\in [0,T^*]$ and $z\in \R$, 
\begin{equation}\label{eq:near_interface}
\l|\P^\epsilon_z\l[\Vote (\v{B}(t))=1\r]-\tfrac{1}{2}\r|\leq \tfrac{5}{12},
\end{equation}
and let $w\in \R$ with $|z-w|\leq c_1(1)\epsilon |\log \epsilon|$. Then 
\begin{equation} \label{eq:high_deriv}
\l|\P^\epsilon_z \l[\Vote(\v{B}(t))=1\r]-\P^\epsilon_w\l[\Vote(\v{B}(t))=1\r]\r|\geq \frac{|z-w|}{48 c_1(1)\epsilon |\log \epsilon |}.
\end{equation}
\end{cor}
\begin{proof}
Consider first the case $0\leq z \leq w$. By analogy with~(\ref{branchid}),
let $\P_y^t$ denote $\P^\epsilon_y\l[\Vote(\v B(t))=1\r]$.
By Theorem~\ref{thm:BBMone} and \eqref{eq:near_interface} we have that
\begin{equation} \label{eq:high_der_lower}
\P^t_{c_1(1)\epsilon |\log \epsilon|}-\P^t_z \geq 1-\epsilon - \tfrac{11}{12}\geq \tfrac{1}{24}.
\end{equation}
Let $\eta:=w-z$. For $j\in \N$, applying Proposition~\ref{prop:deriv1} $j$ times gives that 
$$\P^t_{(j+1)\eta +z}-\P^t_{j\eta +z}\leq \P^t_w-\P^t_z.$$ 
It follows that
\begin{align} \label{eq:high_der_upper}
\P^t_{c_1(1)\epsilon |\log \epsilon|}-\P^t_z 
&\leq \sum_{j=0}^{\lceil \eta^{-1}(c_1(1)\epsilon |\log \epsilon | -z)\rceil-1}(\P^t_{(j+1)\eta +z}-\P^t_{j\eta +z}) \notag \\
&\leq (\eta^{-1}(c_1(1)\epsilon |\log \epsilon |)+1) (\P^t_w-\P^t_z).
\end{align}
Combining~\eqref{eq:high_der_lower} and~\eqref{eq:high_der_upper}, 
$$
\P^t_w-\P^t_z \geq \frac{|z-w|}{24 (c_1(1)\epsilon |\log \epsilon |+|z-w|)}\geq \frac{|z-w|}{48 c_1(1)\epsilon |\log \epsilon |}.
$$
The corresponding result for $0\leq w \leq z$ follows by symmetry (exchanging the roles of $w$ and $z$). The case $z\leq 0$ then follows by the symmetry in \eqref{prob_symmetry}.
\end{proof}

\subsection{A coupling argument}
\label{sec:CFsec}

The second important ingredient in our proof of Theorem~\ref{thm:BBMtwo}
will be a coupling between $d(W_s,t-s)$ (the signed distance from 
a $\dim$-dimensional Brownian motion $W_s$ to 
$\Gamma_{t-s}$, which evolves according to (backwards in time) mean curvature flow)
and a one-dimensional Brownian motion, at least when $W_s$ is
close to $\Gamma_{t-s}$. The proof requires some regularity properties
of the mean curvature flow that we record in this subsection. These 
rest on the assumptions {\Ca}-{\Cc}.

We write 
$\dot{d}$ for the time derivative of $d$.
Let $T^*\in(0,\mathscr{T})$. 
In this case, we have:
\begin{enumerate}
	\item \label{property 1}
	There exists $c_0>0$ such that for all $t\in [0,T^*]$ and 
$x\in \{y:|d(y,t)|\leq c_0\}$, we have 
\begin{equation} \label{eq:unitgrad}
|\nabla d(x,t)|=1.
\end{equation}
Moreover, $d$ is a $C^{\alpha,\frac{\alpha}{2}}$ function in $\{(x,t): |d(x,t)|\leq c_0, t\leq T^*\}$.
\item Viewing $\v{n}=\nabla d$ as the positive normal direction, 
for $x\in \Gamma_t$, the normal velocity of $\Gamma_t$ at 
$x$ is $-\dot{d}(x,t)$, and the curvature of $\Gamma_t$ at $x$ is 
$-\Delta d(x,t)$. Thus,~\eqref{eq:cf_pre} becomes
\begin{equation} \label{eq:cf}
\dot{d}(x,t)=\Delta d(x,t)
\end{equation}
for all $x$ such that $d(x,t)=0$.
\item There exists $C_0>0$ such that for all $t\in[0,T^*]$ and $x$ such that $|d(x,t)|\leq c_0$,
\begin{equation} \label{eq:smoothcf}
\l|\nabla \l(\dot{d}(x,t)-\Delta d(x,t)\r)\r|\leq C_0. 
\end{equation}
\item There exist $v_0,V_0>0$ such that for all $t\in[0,T^*-v_0]$ and all $s\in [t,t+v_0]$,
\begin{equation} \label{eq:t_lipschitz}
|d(x,t)-d(x,s)|\leq V_0(s-t).
\end{equation}
\end{enumerate}
Properties 1 and 2 above come from \cite{chen:1992} (equations~(2.9),~(2.10) and 
Proposition~2.1) and 3 and 4 follow easily from the fact that 
$\sup_{u\in S^1,t\leq T^*}|\Gamma_t(u)|<\infty$ and the regularity of 
$d$ provided by 1.

The first property means that, for each $t\geq 0$, the region 
$\{x:d(x,t)\leq c_0\}$ is not self-intersecting i.e.~for each 
$x$ it contains, the ball $\{z:|z-x|\leq d(x,t)\}$ intersects 
$\Gamma_t$ at precisely one point.
Evidently this cannot hold, for example, as the flow collapses to a
point, which is why we work up to time $T^*<\mathscr{T}$.
Broadly speaking, the first two properties characterize mean
curvature flow in terms of the function $d$. 

A key ingredient of our proof of Theorem \ref{thm:BBMtwo} is the following coupling argument. 

\begin{prop} \label{prop:coupling1} 
Let $(W_s)_{s \geq 0}$ denote a $\dim$-dimensional Brownian motion
started at $x\in \R^{\dim}$. Suppose that $t\leq T^*$, $\beta \leq c_0$ 
and let 
$$T_\beta = \inf \l(\{ s \in [0,t):|d(W_s,t-s)|\geq \beta \}\cup \{t\}\r).$$
Then we can couple $(W_s)_{s \geq 0}$ with a one-dimensional 
Brownian motion $(B_s)_{s\geq 0}$ started from $z=d(x,t)$ in such a way 
that for $s\leq T_\beta$,
$$B_s-C_0\beta s\leq d\l(W_s, t-s\r)\leq B_s+C_0\beta s. $$
\end{prop}
\begin{proof}
By It\^{o}'s formula, we have that for $s\leq t$
$$ d\l(W_s,t-s\r)=\int_0^s A_u\,du+B_s,$$
where 
\begin{align*}
A_u&=-\dot{d}\l(W_u,t-u\r)+\Delta d\l(W_u,t-u\r)\\
B_s&=\sum\limits_{i=1}^\dim\int_0^s \frac{\p }{\p x_i}d(W_u,t-u)dW_u^{(i)}.
\end{align*}
We will handle $A_u$ and $B_s$ in turn.

For each $u\in[0,T_\beta]$ there exists some $x_u\in\R^{\dim}$ such that $|x_u-W_u|\leq\beta$, and $d(x_u,t-u)=0$. By~\eqref{eq:cf} we have $-\dot{d}(x_u,t-u)+\Delta d(x_u,t-u)=0$. Since $\beta\leq c_0$, by~\eqref{eq:smoothcf} we have that, for $x$ on the line segment connecting $x_u$ to $W_u$, the gradient of $-\dot{d}(x,t-u)+\Delta d(x,t-u)$ is bounded by $C_0$. We thus obtain
$$|A_u|\leq C_0\beta.$$
Since $\beta\leq c_0$, it follows by~\eqref{eq:unitgrad} and L\'{e}vy's characterisation (recall that our Brownian motions run at rate $2$) that $(B_s)_{0\leq s\leq T_\beta}$ is a (stopped) Brownian Motion. This completes the proof.
\end{proof}

\begin{remark}
Proposition~\ref{prop:coupling1} provides a probabilistic parallel to one of the key tools used in the classical study of (mean) curvature flow; approximating the movement of the interface locally (in space and time) by a particular one dimensional standing wave.  
\end{remark}

\subsection{Majority voting in BBM, for $\dim\geq2$} 
\label{sec:BBMtwo}

Recall the notation introduced in Section~\ref{subsec:dual_bbm} for 
ternary branching Brownian motion in dimension $\dim \geq 2$. 
For $x\in \R^\dim $, we write $\P ^\epsilon _x$ for the probability measure under which $(\v{W}(t),t \geq 0)$ has the law of ternary branching Brownian motion in $\R^\dim $ with branching rate $1/\epsilon ^2$ started from a single particle at location $x$ at time $0$.
We use $\E ^\epsilon _x$ for the corresponding expectation.
We also write $P_x$ for the probability measure under which $(W_t)_{t \geq 0}$ has the law of a $\dim$-dimensional Brownian motion started at $x$,
and $E_x$ for the corresponding expectation. As usual the notation $B$ (resp.~$\v{B}$) refers to a one dimensional (historical branching) Brownian motion and 
$W$ and $\v{W}$ signal dimension $\dim\geq 2$. 


The proof of Theorem~\ref{thm:BBMtwo} is in two parts. First, in 
Section~\ref{sec:d=2_generation} we establish that the interface 
is generated in a time $\delta_{\dim}=\mc{O}(\epsilon^2|\log\epsilon|)$. 
We then, in Section~\ref{sec:d=2_propagation}, use 
Proposition~\ref{prop:coupling1} and Theorem~\ref{thm:BBMone} 
to investigate how the region
around the interface propagates. In order not to interrupt the flow of the
proof of Theorem~\ref{thm:BBMtwo}, the proof of a central lemma is 
deferred to Section~\ref{sec:d=2_propagation_technical}.

Our proof rests on a comparison with the outcome $\Vote(\v{B}(t))$ of 
majority voting
for the one-dimensional historical branching Brownian motion. 
In one dimension we always 
implicitly take
$\Vote = \Vote_{p_0}$ with
 $p_0(x)=\1\{x\geq 0\}$. We reserve the subscript $p$ for
$\Vote_p(\v{W}(t))$ and we assume that $p$ satisfies~{\Ca}-{\Cc}. 

\subsubsection{Generation of the interface}
\label{sec:d=2_generation}

In this section we prove that, as in $\dim=1$, in dimension $\dim\geq 2$ an interface of width $\mc{O}(\epsilon|\log\epsilon|)$ is generated in time $\mc{O}(\epsilon^2|\log\epsilon|)$.

\begin{prop} \label{prop:d=2_generation}
Let $k\in\N$. Then there exist $\epsilon_{\dim}(k),a_{\dim}(k),b_{\dim}(k)>0$ such that for all $\epsilon\in(0,\epsilon_{\dim})$, if we set
\begin{equation}\label{eq:delta2}
\delta_{\dim}(k,\epsilon) := a_{\dim}(k)\epsilon^2|\log\epsilon| \quad
\textrm{ and }\quad \delta'_{\dim}(k,\epsilon) 
:= (a_{\dim}(k)+k+1)\epsilon^2|\log\epsilon|,
\end{equation}
then for $t\in [\delta_{\dim}, \delta'_{\dim}]$,
\begin{enumerate}
\item for $x$ such that $d(x,t)\geq b_{\dim}\epsilon |\log \epsilon|$, we have $\P^\epsilon_x\l[\Vote _p(\v{W}(t))=1\r]\geq 1-\epsilon^k$;
\item for $x$ such that $d(x,t)\leq -b_{\dim}\epsilon |\log \epsilon|$, we have $\P^\epsilon_x\l[\Vote _p (\v{W}(t))=1\r]\leq \epsilon^k$.
\end{enumerate}
\end{prop}

\begin{proof}
By the same argument as for
Lemma~\ref{ternary_tree}, given $k\in\N$, 
and taking $A(k)$ from Lemma~\ref{g_iteration}, there exist $a_{\dim}(k)$ 
and $\epsilon_{\dim}(k)>0$ such that, for all $\epsilon\in(0,\epsilon_{\dim})$ 
and $t\geq a_{\dim}\epsilon^2|\log \epsilon|$,
\begin{equation}\label{eq:ternary_tree_2d}
\P^\epsilon\l[\mc{T}(\v W (t))\supseteq \mc{T}^{reg}_{A(k)|\log\epsilon|}\r]\geq 1-\epsilon^k.
\end{equation}
It is also easy to obtain a $\dim$-dimensional equivalent of 
Lemma~\ref{all_particles_bound}, with essentially the same proof 
(using a tail bound on a $\dim$-dimensional normal 
distribution instead of one dimensional). That is, given $k\in\N$, there 
exist $d_{\dim}(k)$, $\epsilon_{\dim}(k)$ such that for all $\epsilon \in (0,\epsilon_\dim)$, for 
$t\in [\delta_{\dim}, \delta'_{\dim}]$, 
\begin{equation}\label{eq:all_particles_bound_2d}
\P^\epsilon _x\l[\exists i\in N(t): |W_i(t)-x|\geq d_{\dim}\epsilon|\log\epsilon|\r]\leq \epsilon^k.
\end{equation}
We set $b_{\dim}(k)=2d_{\dim}(k)$.

By~\eqref{eq:t_lipschitz} there exist $v_0,V_0>0$ such that for 
$t\leq v_0$, and any $x\in \R^{\dim}$, we have $|d(x,0)-d(x,t)|\leq V_0t.$ 
Reducing $\epsilon_{\dim}$ if necessary, for 
$\epsilon \in (0,\epsilon_{\dim})$ we have $\delta'_{\dim} \leq v_0$.
Thus, if $\epsilon\in(0,\epsilon_{\dim})$, $t\in [\delta_\dim, \delta'_\dim]$ 
and $x$ is such that $d(x,t)\geq b_{\dim}\epsilon |\log \epsilon |$ and 
$|W_i(t)-x|\leq d_{\dim} \epsilon |\log \epsilon |$ then combining with the 
triangle inequality and~(\ref{eq:t_lipschitz}),
\begin{align*}
d\l(W_i(t),0\r)&\geq d\l(x,t\r)-|d\l(x,t\r)-d\l(W_i(t),t\r)|-|d\l(W_i(t),t\r)-d\l(W_i(t),0\r)|\\
&\geq b_{\dim}\epsilon |\log \epsilon|-d_{\dim}\epsilon |\log \epsilon|-V_0\delta '_\dim\\
&=\frac{1}{2}b_{\dim}\epsilon |\log \epsilon|-V_0(a_{\dim}+k+1)\epsilon^2 |\log \epsilon|.
\end{align*}
Therefore, reducing $\epsilon_{\dim}$ if necessary, in this case we have that
\begin{equation*}
d(W_i(t),0)\geq\tfrac{1}{4}b_{\dim}\epsilon |\log \epsilon |.
\end{equation*}
Applying {\Cb} and {\Cc},
\begin{align}
p(W_i(t))
&\geq \tfrac{1}{2}+\gamma\l(\tfrac{1}{4}b_{\dim} \epsilon |\log \epsilon | \wedge r\r)\notag\\
&\geq \tfrac{1}{2}+\epsilon, \label{eq:p_control}
\end{align}
where we again reduce $\epsilon_{\dim}>0$ (if necessary), to ensure that $\epsilon<\gamma r$, $\epsilon<\frac{\gamma}{4}b_{\dim}\epsilon|\log\epsilon|$ for $\epsilon\in(0,\epsilon_{\dim})$. 

Exactly as in the proof of Theorem~\ref{thm:BBMone},
we can now combine~(\ref{eq:ternary_tree_2d}), 
(\ref{eq:all_particles_bound_2d}) and~(\ref{eq:p_control})
to deduce that for
$\epsilon\in(0,\epsilon_{\dim})$, $t\in [\delta_\dim,\delta'_\dim]$ and 
$x$ such that $d(x,t)\geq b_{\dim}\epsilon|\log \epsilon|$,
$$\P^\epsilon _x\l[\Vote_p(\v W (t))=1\r] \geq 1-3\epsilon^k.$$
The proof of the second statement is analogous.
\end{proof}

\subsubsection{Propagation of the interface and proof of 
Theorem~\ref{thm:BBMtwo}}
\label{sec:d=2_propagation}

We now turn to the propagation of the interface region.
Our immediate goal is to establish that, for suitably chosen (large) $K_1$ and $K_2$, and for all sufficiently small $\epsilon>0$ we have
$$\P^\epsilon_x\l[\Vote_p(\v W (t))=1\r]\approx \P^\epsilon_{d(x,t)+K_1e^{K_2 t}\epsilon|\log\epsilon|}\l[\Vote(\v B (t))=1\r].$$
This connection between $\v{B}$ and $\v{W}$ is made precise by the following result.

\begin{prop} \label{prop:contra}
Let $l\in\N$ with $l\geq 4$. 
Define $a_\dim(l)$ and $\delta_\dim (l,\epsilon)$ as in Proposition~\ref{prop:d=2_generation}.
There exist $K_1(l),K_2(l)>0$ and $\epsilon_{\dim}(l, K_1, K_2)>0$ such that for all $\epsilon\in(0,\epsilon_{\dim})$ and $t\in[\delta_{\dim}(l,\epsilon),T^*]$ we have 
\begin{equation}\label{eq:upper_final}
\sup\limits_{x\in\R^{\dim}}\Big(\P^\epsilon_x\l[\Vote_p(\v W (t))=1\r]-\P^\epsilon_{d(x,t)+K_1e^{K_2 t}\epsilon|\log\epsilon|}\l[\Vote(\v B(t))=1\r]\Big)\leq \epsilon^l
\end{equation}
and 
\begin{equation}\label{eq:lower_final}
\sup\limits_{x\in\R^{\dim}}\Big(\P^\epsilon_x\l[\Vote_p(\v W (t))=0\r]-\P^\epsilon_{d(x,t)-K_1e^{K_2 t}\epsilon|\log\epsilon|}\l[\Vote(\v B(t))=0\r]\Big)\leq \epsilon^l.
\end{equation}
\end{prop}
The proof of Theorem~\ref{thm:BBMtwo}, which follows easily from Proposition \ref{prop:contra}, is at the end of this subsection. 

Recall that $g:[0,1]\to [0,1]$ is given by $g(p)=3p^2-2p^3$.
It is convenient to extend this definition to a continuous, monotone function $g:\R\to [0,1]$ as follows: 
\begin{equation} \label{g_defn_ext}
g(p)= 
\begin{cases} 0 &\mbox{if } p <0 \\
3p^2-2p^3 &\mbox{if } p \in [0,1] \\
1 & \mbox{if } p>1. \end{cases} 
\end{equation}
At the heart of the proof of Proposition~\ref{prop:contra} 
is the following lemma, whose proof we defer 
to Section~\ref{sec:d=2_propagation_technical}.
\begin{lemma}\label{lem:keylemma_2}
Let $l \in \N$ with $l\geq 4$ and $K_1>0$. There exists $K_2=K_2(K_1,l)>0$ 
and $\epsilon_{\dim}(l,K_1,K_2)>0$ such that for all 
$\epsilon\in(0,\epsilon_{\dim})$, $x\in\R^{\dim}$, 
$s\in [0,(l+1)\epsilon^2 |\log \epsilon |]$ and $t\in[s,T^*]$,
\begin{align}
&E_x\l[g\l(\P^\epsilon_{d(W_s,t-s)+K_1e^{K_2 (t-s)}\epsilon |\log \epsilon |}[\Vote(\v B(t-s))=1]+\epsilon ^l\r) \r]\notag\\
&\hspace{8pc}\leq \tfrac{3}{4}\epsilon^l+E_{d(x,t)}\l[g\l(\P^\epsilon_{B_s +K_1 e^{K_2 t}\epsilon |\log \epsilon |}[\Vote(\v B(t-s))=1]\r)\r]
+\1_{s\leq \epsilon ^3} \epsilon ^l \label{eq:keylemma_2eq}
\end{align}
and
\begin{align}
&E_x\l[g\l(\P^\epsilon_{d(W_s,t-s)-K_1e^{K_2 (t-s)}\epsilon |\log \epsilon |}[\Vote(\v B(t-s))=0]+\epsilon ^l\r) \r]\notag\\
&\hspace{8pc}\leq \tfrac{3}{4}\epsilon^l+E_{d(x,t)}\l[g\l(\P^\epsilon_{B_s -K_1 e^{K_2 t}\epsilon |\log \epsilon |}[\Vote(\v B(t-s))=0]\r)\r]
+\1_{s\leq \epsilon ^3} \epsilon ^l. \label{eq:keylemma_2eq_opp}
\end{align}
\end{lemma}

\begin{proof}[Of Proposition~\ref{prop:contra}]
Take $K_1=b_\dim (l)+c_1(l)$ where $b_\dim$ is as defined in 
Proposition~\ref{prop:d=2_generation} and $c_1$ is as defined in 
Theorem~\ref{thm:BBMone}.
Let $K_2=K_2(K_1,l)$, as defined in Lemma~\ref{lem:keylemma_2}.
Take $\epsilon_\dim >0$ sufficiently small that 
Theorem~\ref{thm:BBMone}, Proposition~\ref{prop:d=2_generation} and Lemma~\ref{lem:keylemma_2} 
apply for $\epsilon \in (0,\epsilon_\dim)$.
We begin by observing that for $\epsilon \in (0, \epsilon_\dim)$, 
$t\in [\delta_\dim, \delta'_\dim]$ (where $\delta'_\dim$ is defined 
in~\eqref{eq:delta2}), and $x\in \R^\dim$,
\begin{equation}\label{eq:low_t}
\P^\epsilon_x\l[\Vote_p(\v W(t))=1\r]\leq\P^\epsilon_{d(x,t)+K_1e^{K_2 t}\epsilon|\log\epsilon|}\l[\Vote(\v B(t))=1\r]+ \epsilon^l.
\end{equation}
To see this, note that if $d(x,t)\leq - b_\dim(l)\epsilon |\log \epsilon |$, 
then by Proposition~\ref{prop:d=2_generation}, 
$\P^\epsilon_x\l[\Vote_p(\v W(t))=1\r]\leq \epsilon^l$. 
On the other hand, if $d(x,t)\geq - b_\dim(l)\epsilon |\log \epsilon |$, then $d(x,t)+K_1e^{K_2 t}\epsilon|\log\epsilon|\geq c_1(l)\epsilon |\log \epsilon |$, and so, by Theorem~\ref{thm:BBMone}, \eqref{eq:low_t} holds (since the right
hand side of \eqref{eq:low_t} is $\geq 1$).

We are left with the case $t\in [\delta'_\dim,T^*]$.
We assume, aiming for a contradiction, that 
there exists $t\in[\delta'_{\dim},T^*]$ such that, for some $x\in\R^\dim$, 
$$\P^\epsilon_x\l[\Vote_p(\v W (t))=1\r]-\P^\epsilon_{d(x,t)+K_1e^{K_2 t}\epsilon|\log\epsilon|}\l[\Vote(\v B(t))=1\r]
>\epsilon^l.$$
Let $T'$ be the infimum of the set of such $t$. Choose 
\begin{equation} \label{eq:star_contra}
T\in [T',\min(T'+\epsilon^{l+3},T^*)]
\end{equation}
which is in the set of such $t$. Hence, there exists some $x=x(l,\epsilon)\in\R^\dim$ such that
\begin{equation}\label{eq:for_contra}
\P^\epsilon_x\l[\Vote_p(\v W(T))=1\r]-\P^\epsilon_{d(x,T)+K_1e^{K_2 T}\epsilon|\log\epsilon|}\l[\Vote(\v B(T))=1\r]> \epsilon^l.
\end{equation}
We now seek to show that
\begin{equation}\label{eq:contra}
\P^\epsilon_x\l[\Vote_p(\v W(T))=1\r]\leq \tfrac{7}{8}\epsilon^l+\P^\epsilon_{d(x,T)+K_1e^{K_2T}\epsilon |\log \epsilon |}\l[\Vote(\v B(T))=1\r].
\end{equation}
Since $\frac{7}{8}\epsilon^l<\epsilon^l$, once we obtain equation~\eqref{eq:contra} we have a contradiction to~\eqref{eq:for_contra}, thus completing the proof.

We write $S$ for the time of the first branching event in 
$\v W(T)$
and $W_S$ for 
the position of the initial `ancestor' particle at that time.
We note that by the strong Markov property at time $S\wedge(T-\delta_\dim)$,
\begin{multline} \label{eq:first_branch}
\P^\epsilon_x\l[\Vote_p(\v W(T))=1\r]=
\E^\epsilon_x\l[g(\P^\epsilon_{W_S}\l[\Vote_p (\v W (T-S))=1\r]
\1_{S\leq T-\delta_{\dim}} \r]\\
+\E^\epsilon_x \l[ \P^\epsilon_{W_{T-\delta_{\dim}}}\l[\Vote_p(\v W(\delta_{\dim}))=1\r]\1_{S\geq T-\delta_{\dim}}\r].
\end{multline}
We begin with the second term on the right of~\eqref{eq:first_branch}. 
Since 
$T-\delta_\dim \geq \delta'_\dim-\delta_\dim=(l+1)\epsilon^2 |\log \epsilon |$ and $S\sim \text{Exp}(\epsilon^{-2})$,
\begin{equation} \label{eq:rhs2}
\E^\epsilon_x \l[ \P^\epsilon_{W_{T-\delta_{\dim}}}
\l[\Vote_p(\v W(\delta_{\dim}))=1\r]\1_{S\geq T-\delta_{\dim}}\r]
\leq \P^\epsilon \l[ S\geq (l+1)\epsilon^2 |\log \epsilon |\r]
= \epsilon^{l+1}.
\end{equation}
To bound the first term on the right of~\eqref{eq:first_branch}, 
partition on the event $\{S\leq \epsilon^{l+3}\}$ (which has
probability $\leq\epsilon^{l+1}$):
\begin{align}
&\E^\epsilon_x\l[g(\P^\epsilon_{W_S}
\l[\Vote_p (\v W (T-S))=1\r]\1_{S\leq T-\delta_{\dim}} \r]\notag\\
&\hspace{1pc}\leq \P^\epsilon\l[S\leq\epsilon^{l+3}\r] + 
\E^\epsilon_x\l[g(\P^\epsilon_{W_S}\l[\Vote_p (\v W (T-S))=1\r]
\1_{S\leq T-\delta_{\dim}} \1_{S\geq\epsilon^{l+3}}\r]\notag\\
&\hspace{1pc}\leq \epsilon^{l+1} + \E^\epsilon_x\l[g\l(\P^\epsilon_{d(W_S,T-S)+K_1e^{K_2 (T-S)}\epsilon|\log\epsilon|}\l[\Vote(\v B(T-S))=1\r]+\epsilon^l\r)\1_{S\leq T-\delta_{\dim}}\r]
.\label{eq:rhs1}
\end{align}
The last line follows from the minimality of $T'$ (note that 
if $\epsilon^{l+3}\leq S\leq T-\delta_\dim$, then $T-S \in [\delta_\dim ,T')$ 
by~\eqref{eq:star_contra}) and from monotonicity of $g$.

Conditioning on the value of $S$, since the path of the ancestor particle $(W_\cdot)$ is independent of $S$,
\begin{align} \label{eq:rhs3}
&\E^\epsilon_x\l[g\l(\P^\epsilon_{d(W_S,T-S)+K_1e^{K_2 (T-S)}\epsilon|
\log\epsilon|}\l[\Vote(\v B(T-S))=1\r]+\epsilon^l\r)
\1_{S\leq T-\delta_{\dim}}\r] \notag \\
&\hspace{1cm}\leq \int_0^{(l+1)\epsilon^2 |\log \epsilon |}
\epsilon^{-2}e^{-\epsilon^{-2} s} 
E_x\l[g\l(\P^\epsilon_{d(W_s,T-s)+K_1e^{K_2 (T-s)}\epsilon|\log\epsilon|}\l[\Vote(\v B(T-s))=1\r]+\epsilon^l\r)\r]ds \notag\\
&\hspace{3cm} +\P^\epsilon \l[S\geq (l+1)\epsilon ^2 |\log \epsilon | \r] \notag \\
&\hspace{1cm} \leq \tfrac{3}{4}\epsilon^l+
\int_0^{(l+1)\epsilon^2 |\log \epsilon |} 
\epsilon^{-2}e^{-\epsilon^{-2} s} 
E_{d(x,T)}\l[g\l(\P^\epsilon_{B_s +K_1 e^{K_2 T}\epsilon |\log \epsilon |}
[\Vote(\v B(t-s))=1]\r)\r]ds \notag \\
&\hspace{3cm}+\P^\epsilon \l[S\leq \epsilon ^3 \r] \epsilon ^l +\epsilon^{l+1} \notag \\
&\hspace{1cm} \leq \tfrac{3}{4}\epsilon^l+2\epsilon^{l+1}
+\E^\epsilon_{d(x,T)}\l[g\l(\P^\epsilon_{B_{S'}+K_1e^{K_2 T}\epsilon|\log\epsilon|}\l[\Vote(\v B(T-S'))=1\r]\r)\1_{S'\leq T-\delta_{\dim}}\r].
\end{align}
Here, the second inequality follows by Lemma~\ref{lem:keylemma_2}.
For the final inequality, we write $S'$ for the time of the first branching event in $(\v B(s))_{s \geq 0}$ and $B_{S'}$ for the position of the ancestor at that time, and note that $S'$ has the same distribution as $S$.
The inequality follows since $T\geq \delta'_\dim$ and so 
$T-\delta_\dim \geq (l+1)\epsilon ^2 |\log \epsilon |$.

Putting~\eqref{eq:rhs1}, \eqref{eq:rhs3} and~\eqref{eq:rhs2} into~\eqref{eq:first_branch} we obtain
\begin{align*}
\P^\epsilon_x\l[\Vote_p(\v W (T))=1\r]
&\leq 4\epsilon^{l+1}+\tfrac{3}{4}\epsilon^l+
\E^\epsilon_{d(x,T)}\l[g\l(\P^\epsilon_{B_{S'} +K_1 e^{K_2 T}\epsilon 
|\log \epsilon |}[\Vote(\v B(T-S'))=1]\r)\1_{S'\leq T-\delta_{\dim}}\r] \\
&\leq 4\epsilon^{l+1}+\tfrac{3}{4}\epsilon^l+\P^\epsilon_{d(x,T)+K_1 e^{K_2 T}\epsilon |\log \epsilon |}\l[\Vote(\v B(T))=1\r],
\end{align*} 
where the second line follows by the strong Markov Property for $(\v B (\cdot))$ at time $S'\wedge (T-\delta_\dim)$, in similar style to \eqref{eq:first_branch}.
Reducing $\epsilon_{\dim}$, if necessary, to ensure that 
$\tfrac{3}{4}\epsilon^l+4\epsilon^{l+1}\leq \frac{7}{8}\epsilon^l$ for all $\epsilon\in(0,\epsilon_{\dim})$, we obtain~\eqref{eq:contra}, which completes the proof of \eqref{eq:upper_final}.

By a similar argument,
using~\eqref{eq:keylemma_2eq_opp} in place of \eqref{eq:keylemma_2eq},
we can also deduce \eqref{eq:lower_final}.
\end{proof}

\begin{proof}[Of Theorem~\ref{thm:BBMtwo}]
It suffices to prove the result for sufficiently large $k\in\N$, and in particular we will show it for $k\geq 4$.

We choose $c_{\dim}(k)=c_1(k)+K_1e^{K_2T^*}$. Thus, for any $t\in[\delta_{\dim},T^*]$ and $x\in\R^{\dim}$ such that $d(x,t)\leq -c_{\dim}(k)\epsilon|\log\epsilon|$ we have
$$d(x,t)+K_1e^{K_2t}\epsilon|\log\epsilon|\leq -c_1(k)\epsilon|\log\epsilon|.$$
It follows from Theorem~\ref{thm:BBMone} (reducing $\epsilon_\dim$ if necessary so that $\epsilon<\epsilon_1(k)$) and~\eqref{eq:upper_final} 
that $\P_x\l[\Vote_p(\v W(t))=1\r]\leq 2\epsilon^k$ for such $x$ and $t$. 
Similarly, for $x$ such that $d(x,t)\geq c_{\dim}(k)\epsilon|\log\epsilon|$, 
by Theorem~\ref{thm:BBMone} and~\eqref{eq:lower_final} we have 
$\P_x\l[\Vote_p(\v{W}(t))=0\r]\leq 2\epsilon^k$.
\end{proof}

\subsubsection{Proof of Lemma~\ref{lem:keylemma_2}}
\label{sec:d=2_propagation_technical}

To complete the proof of Theorem~\ref{thm:BBMtwo}, it remains to prove Lemma~\ref{lem:keylemma_2}.
The ideas in the proof are simple, but are easily lost in the notation, so to explain the structure we begin with 
 an outline of the proof of the first 
inequality~\eqref{eq:keylemma_2eq}.
(The proof of~\eqref{eq:keylemma_2eq_opp} goes along essentially the same lines.)

We take a large constant $C$ and consider the cases $|d(x,t)| \geq C \epsilon |\log \epsilon |$ and $|d(x,t)| \leq C \epsilon |\log \epsilon |$ separately.
Since $s=\mc O( \epsilon^2 |\log \epsilon |)$, 
with high probability neither the $\dim$-dimensional Brownian motion $W$ 
nor the one-dimensional $B$ moves a distance
more than $\mc O (\epsilon |\log \epsilon |)$ before time $s$.
Therefore, if $C$ is sufficiently large and $d(x,t) \leq -C \epsilon |\log \epsilon |$, 
Theorem~\ref{thm:BBMone} tells us that the left-hand side of~\eqref{eq:keylemma_2eq} is $\leq \epsilon^{l+1}$;
similarly, if $d(x,t) \geq C \epsilon |\log \epsilon |$ then the right-hand side of \eqref{eq:keylemma_2eq} is $\geq 1$.
This leaves the case of $|d(x,t)| \leq C \epsilon |\log \epsilon |$, in which we apply Proposition~\ref{prop:coupling1} to couple $W_s$ with $B_s$ in such a way that with probability $1-\mc O (\epsilon ^{l+1})$,
$$ d(W_s,t-s)\leq B_s+ \mc O (\epsilon |\log \epsilon |)s. $$
Thus, using monotonicity \eqref{eq:monotonicity}, the left-hand side of~\eqref{eq:keylemma_2eq} 
is bounded above by
$$
\E_{d(x,t)}\l[g\l(\P^\epsilon_{B_s +(K_1e^{K_2 (t-s)}+\mc O (s))\epsilon |\log \epsilon |}[\Vote(\v B(t-s))=1]+\epsilon ^l\r)\r]
+\mc O(\epsilon ^{l+1}). 
$$
If $|p-\frac{1}{2}|\geq \frac{7}{18}$,
we can use that $|g'(p)|\leq 2/3$ to 
pull the $\epsilon^l$ outside the argument of $g$ and then use monotonicity again
to recover~\eqref{eq:keylemma_2eq}. 
The 
difficulty is that close to $p=\tfrac{1}{2}$, we have $g'(p)>1$.
In the case
$\P_{B_s +(K_1e^{K_2 (t-s)}+\mc O (s))\epsilon |\log \epsilon |}
[\Vote(\v{B}(t-s))=1]\approx \frac{1}{2}$, we instead choose $K_2\gg 0$,
and use the lower bound on the `slope of the interface' given by
Corollary~\ref{lem:high_deriv} to estimate the increment in
$\P_z^\epsilon[\Vote(\v{B}(t-s))=1]$ when we replace 
$z+(K_1e^{K_2(t-s)}+\mc O(s))\epsilon |\log\epsilon|$ by 
$z+K_1e^{K_2t}\epsilon |\log\epsilon|$. 

The remainder of this subsection contains the formal proof.

\begin{proof}[Of Lemma~\ref{lem:keylemma_2}]
We begin by proving~\eqref{eq:keylemma_2eq}.
For the duration of the proof, for $u\geq 0$ and $z\in \R$ we write
$$\Q^{\epsilon, u}_{z}=\P^\epsilon_{z}\l[\Vote(\v B(u))=1\r].$$
Recall $C_0$ and $c_1(k)$ from~\eqref{eq:smoothcf} and Theorem~\ref{thm:BBMone} respectively. Let
\begin{equation} \label{R_defn}
R=2 c_1(l)+4(l+1)\dim +1.
\end{equation}
Fix $K_2$ such that
\begin{equation} \label{eq:K_2_cond}
K_1(K_2-C_0)-C_0 R=c_1(1).
\end{equation}
Let $\epsilon_\dim = \epsilon_1(l)$ where $\epsilon_1(l)$ is defined in Theorem~\ref{thm:BBMone}.

First we need an estimate for the probability that a $\dim$-dimensional Brownian motion moves further than $\sim \epsilon |\log \epsilon |$ in time $s$ (recall that $s\leq (l+1)\epsilon^2|\log\epsilon|$). 
Let
$$
A_x=\l\{\sup_{u\in [0,s]} |W_u-x|\leq 2(l+1)\,\dim \epsilon |\log \epsilon |\r\}.
$$
Then bounding $|W_u|$ by the sum of the moduli of $\dim$ one-dimensional Brownian motions and using the reflectional symmetry of one dimensional Brownian motion,
\begin{align} \label{eq:prob_ac_est}
P_x\l[A_x^c\r]&\leq 2\dim P_0 \l[\sup_{u\in [0,s]} B_u >
2(l+1)\epsilon |\log \epsilon | \r] \notag\\
&\leq 4\dim P_0 \l[B_1 >2((l+1) |\log \epsilon |)^{1/2} \r] \notag \\
&\leq 4\dim \epsilon ^{l+1}.
\end{align}
Here, since $s\leq (l+1)\epsilon ^2 |\log \epsilon |$
the second line follows by the reflection principle.
The last line follows using the tail bound
$\P[B_1\geq x]\leq e^{-x^2/4}$.

As advertised, we now consider the following three cases:
\begin{itemize}
\item[(i)] $d(x,t) \leq -\l(2c_1(l)+2(l+1)\dim+K_1e^{K_2(t-s)}\r)\epsilon |\log \epsilon |$,
\item[(ii)] $d(x,t) \geq \l(2c_1(l)+2(l+1)\dim+K_1e^{K_2(t-s)}\r)\epsilon |\log \epsilon |$,
\item[(iii)] $|d(x,t)| \leq \l(2c_1(l)+2(l+1)\dim+K_1e^{K_2(t-s)}\r)\epsilon |\log \epsilon |$.
\end{itemize}
The third case corresponds to $x$ being close to the interface at time $t$. The first two cases correspond to $x$ falling (sufficiently far) inside or outside of the interface.

Case (i): Recall that by~\eqref{eq:t_lipschitz} there exist $v_0,V_0>0$ such that if $s\leq v_0$ and $x\in \R^\dim$ then
\begin{equation} \label{eq:dble_star_t_lip}
|d(x,t)-d(x,t-s)|\leq V_0s.
\end{equation}
We reduce $\epsilon_\dim$, if necessary, to ensure that for 
$\epsilon \in (0, \epsilon_\dim)$ we have $(l+1)\epsilon ^2 |\log \epsilon | \leq v_0$.
Then if the event $A_x$ occurs,
\begin{align}
&d(W_s,t-s)+K_1 e^{K_2 (t-s)} \epsilon |\log \epsilon |\notag\\
&\hspace{1cm}\leq -(2c_1(l)+2(l+1)\dim )\epsilon |\log \epsilon |+|d(W_s,t-s)-d(x,t)|\notag\\
&\hspace{1cm}\leq -(2c_1(l)+2(l+1)\dim )\epsilon |\log \epsilon |+|d(x,t)-d(x,t-s)|+|W_s-x|\notag\\
&\hspace{1cm}\leq -2c_1(l)\epsilon |\log \epsilon |+V_0(l+1)\epsilon ^2 |\log \epsilon |.\label{eq:case_1_start}
\end{align}
Here, the second line follows from being in case (i) and the third follows from the triangle inequality. The final line then follows from \eqref{eq:dble_star_t_lip} and that $s\leq (l+1)\epsilon ^2 |\log \epsilon |$, and since $A_x$ occurs.

Reducing $\epsilon_\dim$, if necessary, from \eqref{eq:case_1_start} we have
$$
d(W_s,t-s)+K_1 e^{K_2 (t-s)} \epsilon |\log \epsilon |\leq -c_1(l)\epsilon |\log \epsilon |.
$$
Therefore
\begin{align*}
E_x \l[g\l(\Q^{\epsilon, t-s}_{d(W_s,t-s)+K_1 e^{K_2 (t-s)} \epsilon |\log \epsilon |} +\epsilon ^l\r)\r]
&\leq E _x \l[g(\epsilon ^l +\epsilon ^l) \1_{A_x} \r]+P_x\l[A_x^c\r]\\
&\leq 6\epsilon ^{2l}+4\dim \epsilon ^{l+1}.
\end{align*}
Here the first inequality follows by Theorem~\ref{thm:BBMone} and 
the second inequality by the definition of $g$ in~\eqref{g_defn_ext} 
and by~\eqref{eq:prob_ac_est}.
Again reducing $\epsilon_\dim$ if necessary, for $\epsilon \in (0,\epsilon_\dim)$ we have
$$
E_x [g(\Q^{\epsilon, t-s}_{d(W_s,t-s)+K_1 e^{K_2 (t-s)} \epsilon |\log \epsilon |} +\epsilon ^l)]\leq \tfrac{3}{4}\epsilon ^l,
$$
and so~\eqref{eq:keylemma_2eq} holds in this case.

Case (ii): In this case, we have that $d(x,t)\geq (c_1(l) +2(l+1))\epsilon |\log \epsilon |$. A similar argument to that used for~\eqref{eq:prob_ac_est} gives us that
\begin{equation} \label{eq:star_key_lem}
P_{d(x,t)}\l[ B_s \leq c_1(l) \epsilon |\log \epsilon | \r]\leq  \epsilon ^{l+1}.
\end{equation}
It follows that in this case
\begin{align*}
E_{d(x,t)}\l[g\l(\Q^{\epsilon, t-s}_{B_s+K_1 e^{K_2 t} \epsilon |\log \epsilon |}\r) \r]
&\geq E_{d(x,t)}\l[g\l(\Q^{\epsilon, t-s}_{B_s+K_1 e^{K_2 t} \epsilon |\log \epsilon |}\r) \1\{B_s\geq c_1(l)\epsilon |\log \epsilon |\}\r]\\
&\geq g(1-\epsilon ^l)-\epsilon ^{l+1}\\
&\geq 1-3\epsilon^{2l}-\epsilon ^{l+1},
\end{align*}
where the second line follows by Theorem~\ref{thm:BBMone} 
and~\eqref{eq:star_key_lem} and the last line by the definition of $g$ 
in~\eqref{g_defn_ext}.
Again reducing $\epsilon_\dim$ if necessary, for $\epsilon\in (0,\epsilon_\dim)$ we have
$$
E_{d(x,t)}\l[g(\Q^{\epsilon, t-s}_{B_s+K_1 e^{K_2 t} \epsilon |\log \epsilon |}) \r]
\geq 1- \tfrac{3}{4} \epsilon^l
$$
and so~\eqref{eq:keylemma_2eq} holds in this case.

Case (iii): We now turn to the case in which $x$ is close to the interface. If the event $A_x$ occurs, for $u\in [0,s]$ we have
\begin{align*}
|d(W_u,t-u)|&\leq |W_u-x|+|d(x,t)|+|d(x,t)-d(x,t-u)|\\
&\leq (2c_1(l)+4(l+1)\dim + K_1 e^{K_2(t-s)})\epsilon |\log \epsilon |
+V_0(l+1)\epsilon ^2 |\log \epsilon |,
\end{align*}
where the second line follows by~\eqref{eq:dble_star_t_lip}.
Reducing $\epsilon_\dim$ if necessary, for $\epsilon \in (0,\epsilon_\dim)$ we have
\begin{equation} \label{eq:dble_dagger_key_lem}
|d(W_u,t-u)|\leq (R+K_1 e^{K_2(t-s)})\epsilon |\log \epsilon |,
\end{equation}
where 
$R$ is defined in~\eqref{R_defn}.
We now apply Proposition~\ref{prop:coupling1} with 
\begin{equation} \label{eq:beta_defn}
\beta = (R+K_1 e^{K_2(t-s)})\epsilon |\log \epsilon |.
\end{equation} 
By reducing $\epsilon_\dim$ if necessary, we have for $\epsilon \in (0,\epsilon_\dim)$ that $\beta \leq c_0$.
Define
$$
T_\beta = \inf (\{u \in [0,t):|d(W_u,t-u)| \geq \beta \}\cup \{t\}).
$$
Then by Proposition~\ref{prop:coupling1}, we can couple $(W_u)_{u\geq 0}$ 
with $(B_u)_{u\geq 0}$, a one-dimensional Brownian motion started from 
$d(x,t)$, in such a way that for $u\leq T_\beta$,
\begin{equation} \label{eq:five_star_key_lemma}
d(W_u,t-u)\leq B_u +C_0 \beta u.
\end{equation}
Hence
\begin{align} \label{eq:0_key_lemma}
E_x \l[g(\Q^{\epsilon, t-s}_{d(W_s,t-s)+K_1 e^{K_2 (t-s)}\epsilon |\log \epsilon |}+\epsilon ^l) \r]
&\leq E_{d(x,t)} \l[g(\Q^{\epsilon, t-s}_{B_s+C_0\beta s +K_1 e^{K_2 (t-s)}\epsilon |\log \epsilon |}+\epsilon ^l) \r]+P_x \l[T_\beta \leq s \r] \notag \\
&\leq  E_{d(x,t)} \l[g(\Q^{\epsilon, t-s}_{B_s+C_0\beta s +K_1 e^{K_2 (t-s)}\epsilon |\log \epsilon |}+\epsilon ^l) \r]+4\dim \epsilon ^{l+1}.
\end{align}
Here, the first line follows by~\eqref{eq:five_star_key_lemma}, \eqref{eq:monotonicity} and the monotonicity of $g$. The second line then follows by~\eqref{eq:prob_ac_est} (note that by~\eqref{eq:dble_dagger_key_lem}, if $A_x$ occurs then $T_\beta \geq s$).

Now let 
$$
E=\l\{\l|\Q^{\epsilon, t-s}_{B_s+C_0\beta s +K_1 e^{K_2 (t-s)}\epsilon |\log \epsilon |}-\tfrac{1}{2}\r|\leq \tfrac{5}{12}\r\}.
$$
We shall consider the cases $E$ and $E^c$ separately to bound the right hand side of~\eqref{eq:0_key_lemma}.  

Consider first when the event $E$ occurs. Note that by the definition of $\beta$ in~\eqref{eq:beta_defn},
\begin{align} \label{eq:K2_conseq}
K_1 e^{K_2 t} \epsilon |\log \epsilon |-\l(C_0\beta s +K_1 e^{K_2 (t-s)}\epsilon |\log \epsilon |\r)
&= \l(K_1 e^{K_2(t-s)}(e^{K_2 s}-1-C_0 s)-C_0 Rs \r)\epsilon |\log \epsilon | \notag\\
&\geq \l( K_1(K_2-C_0)-C_0 R \r)s \epsilon |\log \epsilon | \notag\\
&= c_1(1) s \epsilon |\log \epsilon |,
\end{align}
where the second line follows since $K_2>0$ and the last line follows by~\eqref{eq:K_2_cond}.
Reducing $\epsilon _\dim$ if necessary so that $\epsilon_\dim <\min(\epsilon_1(1),\tfrac{1}{24})$, for $\epsilon \in (0,\epsilon_\dim)$ we can apply 
Corollary~\ref{lem:high_deriv} with 
$z=B_s+C_0\beta s +K_1 e^{K_2 (t-s)}\epsilon |\log \epsilon |$ and 
$w=z+c_1(1) s \epsilon |\log \epsilon |\leq B_s+K_1 e^{K_2 t} \epsilon |\log \epsilon |$ 
to give that
\begin{equation} \label{eq:1_key_lemma}
\Q^{\epsilon, t-s}_{B_s+C_0\beta s +K_1 e^{K_2 (t-s)}\epsilon |\log \epsilon |} \1_E
\leq 
(\Q^{\epsilon, t-s}_{B_s+K_1 e^{K_2 t}\epsilon |\log \epsilon |}-\tfrac{1}{48}s) \1_E .
\end{equation}

Finally, we consider the case when the event $E^c$ occurs. Recall that $g(p)=3p^2-2p^3$ for $p\in [0,1]$, so $g'(p)=6p(1-p)$.
Hence if $p,\delta \geq 0$ with either $p+\delta \leq \frac{1}{9}$ or  $p \geq \frac{8}{9}$ then
\begin{equation} \label{eq:p_deriv_key}
g(p+\delta) \leq g(p)+\tfrac{2}{3}\delta .
\end{equation}
Reducing $\epsilon _\dim$ if necessary so that $\frac{1}{12}+\epsilon^l <\frac{1}{9}$ for $\epsilon \in (0,\epsilon_\dim)$, we have 
\begin{align} \label{eq:2_key_lem}
g \l(\Q^{\epsilon, t-s}_{B_s+C_0\beta s +K_1 e^{K_2 (t-s)}\epsilon |\log \epsilon |}+\epsilon ^l \r) \1_{E^c}
& \leq \l( g \l(\Q^{\epsilon, t-s}_{B_s+C_0\beta s +K_1 e^{K_2 (t-s)}\epsilon |\log \epsilon |}\r)+\tfrac{2}{3}\epsilon ^l \r) \1_{E^c} \notag \\
& \leq \l( g \l(\Q^{\epsilon, t-s}_{B_s+K_1 e^{K_2 t}\epsilon |\log \epsilon |}\r)+\tfrac{2}{3}\epsilon ^l \r) \1_{E^c},
\end{align}
where the first line follows by~\eqref{eq:p_deriv_key} and the last line by~\eqref{eq:K2_conseq} and monotonicity of $g$.

Putting \eqref{eq:1_key_lemma} and \eqref{eq:2_key_lem} into \eqref{eq:0_key_lemma},
\begin{align*}
E_x \l[g(\Q^{\epsilon, t-s}_{d(W_s,t-s)+K_1 e^{K_2 (t-s)}\epsilon |\log \epsilon |}+\epsilon ^l) \r]
&\leq E_{d(x,t)} \l[g\l(\Q^{\epsilon, t-s}_{B_s+K_1 e^{K_2 t}\epsilon |\log \epsilon |}-\tfrac{1}{48}s +\epsilon^l \r)\1_E \r]\\
&\quad +E_{d(x,t)} \l[\l( g\l(\Q^{\epsilon, t-s}_{B_s+K_1 e^{K_2 t}\epsilon |\log \epsilon |}\r) +\tfrac{2}{3}\epsilon^l \r)\1_{E^c} \r]\\
&\quad +4\dim\epsilon^{l+1}\\
&\leq E_{d(x,t)} \l[g\l(\Q^{\epsilon, t-s}_{B_s+K_1 e^{K_2 t}\epsilon |\log \epsilon |}\r)\r]\\
&\quad +\tfrac{2}{3}\epsilon^l + \epsilon ^l \1 _{s\leq 48 \epsilon ^l}+4\dim\epsilon^{l+1},
\end{align*}
where the last inequality follows in the case $s\leq 48 \epsilon ^l$ since $|g'(p)|\leq \frac{3}{2}$ for all $p \in [0,1]$. 
Reducing $\epsilon_\dim$, if necessary, so that $4\dim\epsilon^{l+1}\leq \frac{1}{12}\epsilon^l$ and $48\epsilon ^l \leq \epsilon^3$ for $\epsilon \in (0,\epsilon _\dim)$ completes the proof of~\eqref{eq:keylemma_2eq}.

The second statement of the lemma, equation \eqref{eq:keylemma_2eq_opp}, is proved by the same argument, considering $\{\Vote(\v B(u))=0\}$ instead of $\{\Vote(\v B(u))=1\}$ and using $d(W_u,t-u)\geq B_u -C_0 \beta u$ for $u\leq T_\beta$ in place of~\eqref{eq:five_star_key_lemma}.
\end{proof}

\section{Proof of Theorem~\ref{thm:slfvs}}
\label{proof of slfvs to cf}

In this section we turn to the proof of our central result, 
Theorem~\ref{thm:slfvs}, which provides convergence, after suitable rescaling,
of the SLFVS started from an appropriate initial condition to the indicator 
function of a region whose boundary evolves according to mean curvature
flow. The proof mimics that of Theorem~\ref{theorem ac to cf}
in exploiting a dual process. However, because of genetic drift, in addition
to branching, individuals in our dual process can coalesce. The duality
relation will once again be with a historical process and expressed
through a majority voting procedure.

\subsection{A branching and coalescing dual for the SLFVS}
\label{duality for SLFVS}

We begin by describing the dual process of branching and coalescing
lineages. It is driven by the same Poisson Point Process of `events'
that drives the SLFVS.
Recall from~\eqref{eq:slfvs_intensity_intro} that
$\Pi^n$ is a Poisson point process on 
$\R_+ \times \R^{\dim} \times (0,\infty)$ with intensity measure 
$$
n dt\otimes n^{\beta} dx\otimes \mu^n(dr).
$$
We also let
$$
u_n = \frac{u}{n^{1-2\beta}}, \qquad\mbox{and}\qquad 
\v{s}_n = 
\frac{1}{\epsilon_n^{2}}\frac{1}{n^{2\beta}}.
$$

\begin{defn}[SLFVS dual]
\label{def:slfvs_dual}
For $n\in\N$, the process $(\mathcal P ^n _t)_{t\geq 0}$ is the
$\bigcup_{l\geq 1}(\R^{\dim})^l$-valued Markov process 
with dynamics defined as follows.

The process is started with a single individual 
$\mathcal P^n_0=x$ and for $t\geq 0$, 
$\mathcal P^n_t = (\xi^n_1(t),\ldots , \xi^n_{N(t)}(t))$ for some $N(t)\in \N$.
At each event $(t,x,r)\in \Pi^n$, independently of all else, 
the event is said to be neutral with probability $1-\v{s}_n$. In this case: 
\begin{enumerate}
\item For each $\xi_i^n(t-)\in \mc{B}_r(x)$, independently mark 
the corresponding 
individual with probability $u_n$;
\item if at least one individual is marked, 
all marked individuals coalesce into a single
offspring individual, whose location is drawn uniformly at random 
from within $\mc{B}_r(x)$.
\end{enumerate}
With the complementary probability $\v{s}_n$, the event is said to be 
selective, in which case:
\begin{enumerate}
\item For each $\xi^n_i(t-)\in \mc{B}_r(x)$, independently mark the corresponding
individual with probability $u_n$;
\item if at least one individual is marked,
all of the marked individuals are replaced by {\em three} offspring
individuals, whose locations are drawn independently and uniformly from
within $\mc{B}_r(x)$.
\end{enumerate}
In both cases, if no individual is marked, then nothing happens.
\end{defn}
\begin{remark}
We have referred to the new individuals created during reproduction events
as `offspring' individuals. From a biological perspective, it would perhaps 
be more
natural to call them `parents' or `potential parents', as forwards in time
they correspond to the locations from which alleles from the parental
generation are sampled.
However, as much of our proof of Theorem~\ref{theorem ac to cf}
will carry over with minimal changes to
the SLFVS setting, we wish to
retain the terminology of the branching Brownian motion of the previous
section.
\end{remark}
The duality relation that we exploit is between the SLFVS and
the {\em historical process}
of branching and coalescing lineages, 
$$\Xi^n(t):=
(\mathcal P^n_s)_{0\leq s\leq t}.$$
We write $\P_x$ for the law of $\Xi^n$ when $\mathcal P^n_0$ is 
the single point $x$ and $\E_x$ for the corresponding expectation. 
For $\v i\in \{1,2,3\}^\N$
with $\v{i}=(i_1,i_2,\ldots)$,
we let $(\xi^n_{\v i}(\cdot))_{0\leq s\leq t}\subseteq\Xi(t)$ denote the $\R^\dim$-valued path 
which jumps to the location of an offspring when the individual in $\mc P^n_s$ at its location is affected by an event, jumping to the $i_k ^{\text{th}}$ offspring when it is affected by its $k^{\text{th}}$ selective event.
We shall refer to $(\xi^n_{\v i}(\cdot))_{0\leq s\leq t}$ as an
ancestral lineage.

The voting procedure on $\Xi^n(t)$ is a minor modification 
of Definition~\ref{vote_defn}.
Let $p:\R^\dim \to [0,1]$ be a fixed function.
Recalling that the set of individuals in $\mc P^n_t$ is 
$\{\xi^n_1(t),\ldots , \xi^n_{N(t)}(t)\}$,
for each $j\leq N(t)$, the individual 
$\xi_j^n(t)$ votes $1$ with probability $p(\xi_j^n(t))$
and otherwise votes $0$; votes from different individuals are independent. 
As we trace backwards in time through $\Xi(t)$,
\begin{enumerate}
\item at each neutral event, all individuals that are marked in the event
adopt the vote of the offspring individual of the event;
\item at each selective event in $\Pi^n$, all individuals 
that are marked in the event adopt the 
majority vote of the votes of the three offspring individuals of the event.
\end{enumerate}
This defines an iterative voting procedure, which runs inwards from 
the `leaves' of $\Xi^n(t)$ to the ancestral individual $\emptyset$. 
\begin{defn}[$\Vote _p $] \label{vote_defn_slfvs}
With the voting procedure described above, we define
$ \mathbb{V} _p(\Xi^n(t)) $
to be the vote associated to the root $\emptyset$. 
\end{defn}

At this point the duality relation between the SLFVS and $\Xi(t)$ is 
easy to guess. However, 
in order to write it down formally, we have to overcome the
fact that the SLFVS will only be defined, as a function, Lebesgue a.e.~and
so we cannot necessarily define 
$w^n_t(x)$ for a fixed point $x\in \R^\dim$.
However, if,
$\psi\in C(\R^\dim)\cap L^1(\R^\dim)$, 
then the function
$$\int_{\R^\dim}\psi(x)w^n_t(x)dx,$$
{\em is} well-defined.

\begin{theorem} \label{thm:slfvs_duality}
The spatial $\Lambda$-Fleming-Viot process with selection driven by $\Pi^n$,
$(w^n_t(x), x\in\R^\dim)_{t\geq 0}$, is
dual to the historical process $(\Xi^n(t))_{t\geq 0}$ in the sense that
for every $\psi\in C(\R^\dim)\cap L^1(\R^\dim)$, we have
\begin{equation}
\E_{p}\bigg[\int_{\R^\dim}  \psi(x)w^n_t(x)\, dx\bigg] 
= \int_{\R^\dim} \psi(x)\E_x\bigg[ \Vote_{p}\big( \Xi^n(t)\big)\bigg]\, dx
= \int_{\R^\dim} \psi(x)\P_x\bigg[ \Vote_{p}\big( \Xi^n(t)\big)=1\bigg]\, dx.
\label{dual formula}
\end{equation}
\end{theorem}
\begin{remark}
Of course, we are abusing notation here:
the expectations on the left and right of this equation are
taken with respect to different measures. The subscripts on the expectations
are the initial values for the processes on each side.
\end{remark}
To see that the result should be true, note that (if it is defined) $w_t^n(x)$ is the probability that an allele sampled from the population at the location $x$ at time $t$ is of type $a$. In order to determine that probability, we trace back until the most recent event that covered the location $x$. With probability $u_n$, the chosen allele was an offspring of the event, in which  case its type can be determined if we know the types of the potential parents of the event. If the event is neutral, the type is that of an allele (the `parent') sampled from a point picked uniformly at random from the affected region at the time of the event; if it is selective, then the type is the `majority vote' of three `potential parents'  sampled uniformly at random from the affected region. In order
to determine the types of the potential parents, we continue to trace backwards in time, following the locations of all potential ancestors until time zero. This gives us the dual process
$\Xi^n(t)$. At that time, each potential ancestor samples its type according to the initial condition $w_0$ at its location. We can then determine $w_t^n(x)$ by working back through $\Xi^n(t)$ using our majority voting procedure.

A formal proof of Theorem~\ref{thm:slfvs_duality} using generators is a simple extension of that of the corresponding duality for the 
spatial $\Lambda$-Fleming-Viot process with genic selection  
in~\cite{etheridge/veber/yu:2014} (and indeed can be extended to 
cover the more general initial conditions for the dual process considered
there) and so is omitted. 

The duality reduces the proof of 
Theorem~\ref{thm:slfvs} to the following analogue
of Theorem~\ref{thm:BBMtwo}. 
\begin{theorem} \label{thm:slfvs_dual}
Take $\sigma^2$ as in \eqref{defn of sigma}.
Suppose that $\beta\in(0,1/4)$ and let 
$\epsilon_n$ be a sequence such that $\epsilon_n\to 0$ and 
$(\log n)^{1/2}\epsilon_n\to\infty$ as $n\rightarrow\infty$.
Assume $p$ satisfies {\Ca}-{\Cc} and define $\mathscr{T}$, $d(x,t)$ as for Theorem \ref{theorem ac to cf};
take $T^*<\mathscr{T}$. Let $k\in\N$. There exist 
$n_*(k)\in \N$, and $a_*(k),d_*(k)\in(0,\infty)$ such that for all 
$n\geq n^*$ and all $t$ satisfying 
$a_* \epsilon_n ^2 |\log \epsilon_n |\leq t\leq T^*$, 
\begin{enumerate}
\item for $x$ such that 
$d(x,\sigma ^2 t)\geq d_* \epsilon_n |\log \epsilon_n|$, 
we have $\P_x\l[\Vote_p (\Xi^n (t))=1\r]\geq 1-\epsilon_n^k$.
\item for $x$ such that 
$d(x,\sigma ^2 t)\leq -d_* \epsilon_n |\log \epsilon_n|$, 
we have $\P_x\l[\Vote_p (\Xi ^n (t))=1\r]\leq \epsilon_n^k$.
\end{enumerate}
\end{theorem}
Before providing a proof of this result, let us explain why it should be 
true. 

First consider 
the motion of a single ancestral lineage $\xi_{\v i}^n(\cdot)$ in $\Xi^n(t)$.
It evolves as a pure jump process which is
homogeneous in both space and time. Write $V_r$ for the volume
of $\mc{B}_r(x)$.
The rate at which the lineage jumps from
$y$ to $y+z$ can be written
\begin{equation}
\label{jump of size z}
m_n(dz)=nu_nn^{\dim\beta}\int_0^{\mc{R}_n}\frac{V_r(0,z)}{V_r}\mu^n(dr)\,dz,
\end{equation}
where $V_r(0,z)$ is the volume of ${\mc B}_r(0)\cap {\mc B}_r(z)$.
To see this, by spatial homogeneity, we may take the lineage to be at the 
origin in $\R^\dim$ before the jump, and then, in order for it to jump to
$z$, it must be affected by an event that covers both $0$ and $z$. If the
event has radius $r$, then the 
volume of possible centres, $x$, 
of such events is $V_r(0,z)$ and so the intensity with which such a centre
is selected is
$n\,n^{\dim\beta}V_r(0,z)\mu^n(dr)$. 
The parental location is
chosen uniformly from the ball $\mc{B}_r(x)$, so the probability that 
$z$ is chosen as the parental location is $dz/V_r$ and the probability
that our lineage is actually affected by the event is $u_n$.
Combining these yields~\eqref{jump of size z}.

The total rate of jumps is
\begin{eqnarray}
\int_{\R^\dim}m_n(dz)&=&\int_0^{\mc{R}_n}nu_n\,n^{\dim\beta}\frac{1}{V_r}
\int_{\R^\dim}\int_{\R^\dim}\1_{|x|<r}\1_{|x-z|<r}dx\,dz\,\mu^n(dr)
\nonumber\\
&=&\int_0^{\mc{R}_n}nu_n\,n^{\dim\beta}V_r\mu^n(dr)\nonumber \\
&=&n^{2\beta} u V_1\int_0^{\mc{R}}r^d\mu(dr),\label{jump rate}
\end{eqnarray}
and the size of each jump is $\Theta(n^{-\beta})$ and so it is no
surprise that in the limit a single lineage will evolve according to
a (time-changed) Brownian motion. To identify the diffusion constant, we
calculate: 
\begin{multline}
\label{identifying sigma}
\frac{1}{2\dim}\int_{\R^\dim}|z|^2m_n(dz)
=\frac{1}{2\dim}\int_{\R^\dim}|z|^2nu_n\int_0^{\mc{R}_n}n^{\dim\beta}
\frac{V_r(0,z)}{V_r}\mu^n(dr)dz\\
=\frac{u}{2\dim}\int_0^{\mc{R}}\int_{\R^\dim}|z|^2\frac{V_r(0,z)}{V_r}dz\mu(dr),
\end{multline}
which is precisely $\sigma^2$ from~\eqref{defn of sigma}.

Note also that a lineage is affected by selective events at rate
\begin{equation}\label{eq:sel_rate_slfvs}
\l(uV_1\int_0^\mc{R}r^d\mu(dr)\r)n^{2\beta}\v{s}_n
=\eta \epsilon_n^{-2},
\end{equation}
where 
$\eta =uV_1\int_0^\mc{R}r^d\mu(dr)$.
Evidently, we can bound the total number of lineages in $\Xi^n(t)$ above by 
the total number in a process in which each lineage, independently, branches at 
rate $\eta \epsilon_n^{-2}$. Since  
$\epsilon_n^{-2}=o(\log n)$, this implies that for any $\delta>0$,
with high probability, there are 
$o(n^\delta)$ pairs of lineages in $\Xi^n(T^*)$. Each such pair is
in the region affected by some event (neutral or selective)
at most $\mc{O}(n)$ times in $[0,T^*]$
and so the chance that we see any coalescence events is 
$o(nu_n^2\,n^\delta)$ for any $\delta>0$. Since
$nu_n^2=n^{4\beta-1}$ 
and $\beta \in (0,1/4)$, for large $n$ we do not expect to
see any coalescence events before time $T^*$. 

Combining the above, the dual is well approximated by a ternary branching
Brownian motion with branching rate $\Theta(\epsilon_n^{-2})$ and 
so it is natural to expect that an equivalent of 
Theorem~\ref{thm:BBMtwo} holds.

\subsection{Majority voting in the {\slfvs}, for $\dim\geq 2$}
\label{sec:slfvs}

The rigorous proof of Theorem~\ref{thm:slfvs}
closely follows that of Theorem~\ref{thm:BBMtwo}. 
In Section~\ref{sec:slfvs_generation}, we focus on generation of 
the interface, which is proved in much the same way as 
Proposition~\ref{prop:d=2_generation}. 
Then, in Section~\ref{sec:slfvs_propagation}, we look at the propagation 
of the interface. We shall see that, since it essentially focusses on a 
single branching event, the argument of Section~\ref{sec:d=2_propagation} 
is sufficiently flexible to adapt to the SLFVS setting.

First we present the additional arguments required in the SLFVS setting.
These stem from the fact that ancestral
lineages in the dual of the SLFVS follow jump processes (which,
when the lineages are too close together, are dependent), and from the
coalescence of ancestral lineages.
In Section~\ref{sec:slfvs_dual_ingredients} we show that (in between
selective events) the motion of a single ancestral lineage
is approximately (time-changed) Brownian motion. 
Then, in Section~\ref{sec:slfvs_indep_after_branching},
we show that, asymptotically, the three families of descendants of
offspring created during a selective event evolve independently
(conditional on their locations at birth).

\begin{remark}
In Sections~\ref{sec:BBMone} and~\ref{sec:BBMtwo} we used subscripts to 
distinguish variables that played the same role in each section,
but had different values; e.g.~$\delta_1$ in~\eqref{delta 1} 
and $\delta_{\dim}$ in~\eqref{eq:delta2}. The corresponding 
quantities in this section will be denoted with
a subscript $*$, for example $\delta_*$ in~\eqref{eq:delta_slfv}.
\end{remark}

\subsubsection{A single lineage}
\label{sec:slfvs_dual_ingredients}

We begin the proof by showing that the trajectory of a single 
lineage is close to that of a Brownian motion. We follow what is now a
familiar argument in the context of spatial $\Lambda$-Fleming-Viot processes
(see for example~\cite{etheridge/freeman/penington/straulino:2015}).

Let $(\xi^n(t))_{t \geq 0}$ be a pure jump process started at 
$x\in \R^\dim$ with rate of jumps from $y$ to $y+z$ given by the 
intensity measure $m^n(dz)$, and let $(W(t))_{t\geq 0}$ be a Brownian 
motion in $\R^\dim$ started at $x$.
\begin{lemma} \label{lem:W_xin_close}
For $t>0$ fixed, there is a coupling of $W$ and $\xi^n$ under which
$$
\P \l[ \l|\xi^n(t)-W(\sigma^2 t)\r| \geq n^{-\beta/6} \r]=
\mc O(n^{-\beta }(t \vee 1)).
$$
\end{lemma}
\begin{proof}
For $i\geq 1$, let $X_i = \xi^n _{i/n^{2\beta}} - \xi^n_{(i-1)/n^{2\beta}}$.
Then $X_1 , X_2, \ldots $ are i.i.d.~with a rotationally symmetric 
distribution and, by~\eqref{identifying sigma},
$\E[|X_1|^2]=2\dim \sigma^2 n^{-2\beta}$.
Moreover, by~\eqref{jump rate}, the number of jumps made by 
$\xi^n$ on the time interval $[0,n^{-2\beta}]$ is Poisson, with mean
$\Theta(1)$, so since each jump has magnitude at most $2 \mc R_n$,
$
\E \l[ |X_1 |^4 \r]  = \mc O (n^{-4\beta})$.
Then by Skorohod's second embedding Theorem, see e.g.~\cite{billingsley:1995},
there is a Brownian motion $W$ started at $x$ and a sequence 
$\upsilon_1, \upsilon_2, \ldots $ of stopping times such 
that setting $\upsilon_0=0$, $(\upsilon_i - \upsilon_{i-1})_{i \geq 1}$ 
are i.i.d.~and
\begin{align*}
W(\upsilon_i)&=\xi(i/n^{2\beta}),\hspace{2pc}
\E[\upsilon_i - \upsilon_{i-1}]=\tfrac{1}{2\dim}\E\l[ |X_1 |^2 \r]=\sigma^2 n^{-2\beta},\hspace{2pc}
\E[(\upsilon_i - \upsilon_{i-1})^2 ] = \mc{O}(n^{-4\beta}).
\end{align*} 
It follows that $\E[\upsilon_{\lfloor tn^{2\beta} \rfloor}]
=\sigma^2 \lfloor tn^{2\beta}\rfloor  n^{-2\beta}$ 
and $\text{Var}[\upsilon_{\lfloor tn^{2\beta} \rfloor}]=\mc{O}(t n^{-2\beta})$.
Hence by Chebychev's inequality,
\begin{equation}\label{eq:sync_time_skemd}
\P\l[|\upsilon_{\lfloor tn^{2\beta} \rfloor}- \sigma^2 t|
\geq n^{-\beta/2}\r]=\mc{O}(t n^{-\beta}).
\end{equation}
Now we have that
\begin{equation} \label{eq:xi_W_bound}
|\xi^n(t)-W(\sigma^2 t)|\leq |\xi^n(t)-\xi^n(\lfloor t n^{2\beta} \rfloor /n^{2 \beta})|
+|W(\upsilon_{\lfloor t n^{2\beta} \rfloor})-W(\sigma^2 t)|.
\end{equation}
To control the first term on the right hand side, observe that
\begin{equation} \label{eq:xi_W_1st}
\P \l[ |\xi^n(t)-\xi^n(\lfloor t n^{2\beta} \rfloor /n^{2 \beta})|
\geq n^{-\beta/6}/2\r]
\leq \E\l[ |X_1|^2 \r] (n^{-\beta/6}/2)^{-2}=\mc O (n^{-5\beta/3}). 
\end{equation}
To control the second term on the right hand side of~\eqref{eq:xi_W_bound}, 
let $Z\sim N(0,1)$, then
\begin{eqnarray} \label{eq:xi_W_2nd}
\P\bigg[|W(\upsilon_{\lfloor t n^{2\beta} \rfloor})-W(\sigma^2 t)|
&\geq& n^{-\beta /6}/2 \bigg]\leq 
\P\l[|\upsilon_{\lfloor tn^{2\beta} \rfloor}- \sigma^2 t|
\geq n^{-\beta/2}\r]\nonumber\\
&&+
\P\l[|\upsilon_{\lfloor t n^{2\beta} \rfloor }-\sigma^2 t| \leq n^{-\beta /2},
\, |W(\upsilon_{\lfloor t n^{2\beta} \rfloor})-W(\sigma^2 t)|
\geq n^{-\beta /6}/2 \r]\nonumber \\
&\leq& \P \l[\sup_{s\in [-n^{-\beta/2},n^{-\beta/2}]} |W(s)-W(0)| 
\geq n^{-\beta/6}/2 \r]+
\mc{O}(t n^{-\beta}).
\nonumber \\
&\leq& 4\, \dim\, \P \l[\sqrt 2 n^{-\beta/4}Z \geq n^{-\beta/6}/2\dim  \r]
+\mc{O}(t n^{-\beta}).
\nonumber \\
&=& \mc O (\exp (-\tfrac{1}{8\dim ^2} n^{\beta/6}) )
+\mc{O}(t n^{-\beta}).
\end{eqnarray}
Here, the second inequality follows by \eqref{eq:sync_time_skemd} and the third inequality follows by bounding the modulus of a 
$\dim$-dimensional Brownian motion by the sum of the moduli of 
$\dim$ one-dimensional Brownian motions and then using the reflection 
principle. 
Combining~\eqref{eq:xi_W_1st} 
and~\eqref{eq:xi_W_2nd} with~\eqref{eq:xi_W_bound} completes the proof.
\end{proof}

Next, we need the asymptotic distribution of an ancestral lineage and its first
branch time (that is the first time that it is affected by a selective event).
\begin{cor} \label{cor:xi_W}
Let $\tau$ be the first branch time of $\Xi^n$. There is a coupling of 
$\Xi^n$ and $W$ under which $\tau$ and $W$ are independent,
$\tau \sim \text{Exp}(\eta \epsilon_n^{-2} )$ where 
$\eta =u V_1\int_0^\mc{R}r^d\mu(dr)$,
and for $i=1,2,3$,
$$
\P\l[\xi^n_i (\tau)-W(\sigma ^2 \tau)| \geq 3n^{-\beta/ 6} \r]=
\mc O (n^{-\beta}).
$$
\end{cor}
\begin{proof}
The distribution of $\tau$ follows immediately from~\eqref{eq:sel_rate_slfvs}.

Now consider any ancestral lineage $\xi^n\subseteq\Xi^n$. 
By the thinning property of Poisson processes, at any time $t>0$,
we can write
$\xi^n_t=\xi^{n,\tt{sel}}_t+\xi^{n,\tt{neu}}_t$, where $\xi^{n,\tt{sel}}_t$
and $\xi^{n,\tt{neu}}_t$ are independent pure jump processes with jump
intensities $\v{s}_nm_n(dz)$ and $(1-\v{s}_n)m_n(dz)$ respectively,
and taking $\tau$ to be the first jump time of $\xi^{n,\tt{sel}}$,
$\xi^{n,\tt{neu}}_t$ is independent of $\tau$. Using 
Lemma~\ref{lem:W_xin_close} with $(1-\v{s}_n)m_n(dz)$ in place of $m_n(dz)$,
we can couple $\xi^{n,\tt{neu}}$ with a Brownian motion $W$ in such a way that for any $t>0$,
for any $t>0$,
$$\P[|\xi^{n,\tt{neu}}_t-W(\sigma^2(1-\v{s}_n)t)|\geq n^{-\beta/6}]\leq 
\mc O (n^{-\beta}(t\vee 1)).$$
Since $\v{s}_n=o (\log n/n^{2\beta})$, using Chebyshev's inequality, 
$$\P[|W(\sigma^2 t)-W(\sigma^2(1-\v{s}_n)t)|\geq n^{-\beta/6}]
= o \left(\frac{\log n}{n^{2\beta}}n^{\beta/3}(t\vee 1)\right),$$
and so using the triangle inequality
$$
\P \l[|\xi^n (\tau-)-W(\sigma^2 \tau )| \geq 2n^{-\beta/6 } \bigg| \tau \r]=\mc O(n^{-\beta} (\tau \vee 1)).
$$
Since $\E[\tau]=\Theta (\epsilon_n^2)=o(1)$, and for $i=1,2,3$, $|\xi^n_i(\tau)-\xi^n_1(\tau-)|\leq 2 \mc R_n=2n^{-\beta}\mc R$ the result follows.
\end{proof}

\subsubsection{Independence after branching}
\label{sec:slfvs_indep_after_branching}

We now define a modification of $\Xi^n(t)$ which we denote by
$\Psi^n(t)$ in which lineages 
evolve independently after branching (so, in particular, do not coalesce) 
and then
show that $\Xi^n(t)$ and $\Psi^n(t)$ can be coupled in such a way
that they coincide with high probability.

\begin{defn}[Branching jump process]
\label{def:slfvs_dual_no_coal}
For given $n\in \N$ and starting point $x\in \R^\dim$, 
$(\Psi^n(t), t \geq 0)$ is the 
historical process of the branching random walk which is described as follows.
\begin{enumerate}
\item
Each individual has an independent exponential lifetime with 
parameter $\eta\epsilon_n^{-2}$. 
\item During
its lifetime, each individual, independently, evolves according to 
a pure jump process with jump intensity $(1-\v{s}_n)m_n(dz)$. 
\item
At the end of its lifetime an individual branches into three offspring.
\item The locations of the offspring are determined as follows.
For each branching event, independently, pick $r\in (0, \mc{R}_n]$
according to $r^\dim \mu^n(dr)/\int_0^{\mc{R}_n}r^\dim \mu^n(dr)$. If the parent is at
the point $z\in\R^\dim$, then each of the three offspring, independently, 
samples its location uniformly from $B_r(z)$.
\end{enumerate}
\end{defn}
\begin{remark}
Note that the only difference between the distributions of $\Xi^n$ and $\Psi^n$ is that in $\Psi^n$, lineages evolve independently after branching, whereas in $\Xi^n$, two distinct lineages may be hit by the same event in $\Pi_n$.
\end{remark}

We define $\Vote_p(\Psi^n(t))$ in the usual way (as in Definition~\ref{vote_defn}): a leaf at location
$\psi_i(t)\in\R^\dim$
votes $1$ with probability $p(\psi_i(t))$, otherwise it votes zero, and votes from different leaves are independent; working
back through the tree
an individual adopts the vote of the majority of its offspring and 
$\Vote_p(\Psi^n(t))$ is the resultant vote at the root.

\begin{lemma}\label{lem:P=D}
Let $T^*\in(0,\infty)$, $k\in\N$ and $z\in\R^{\dim}$. 
There exists $n_*\in \N$ such that for all $n\geq n_*$, there is a 
coupling of $\Xi^n$ started from $z$ and $\Psi^n$ started from $z$ such 
that with probability at least $1-\epsilon_n^k$ we have 
$$
\Xi^n(T^*)=\Psi^n(T^*).
$$
\end{lemma}

The remainder of this section is devoted to proof of Lemma~\ref{lem:P=D}. 
To do so, we consider a slightly different description of
the dual of the {\slfvs}, 
which will preserve the distribution of $\Xi^n$. 
\begin{defn}[Pre-emptive SLFVS dual]
\label{def:slfvs_dual_preemptive} 
For $n\in\N$, the process $(\tilde{\mathcal P} ^n _t)_{t\geq 0}$ is a 
$\bigcup_{l\geq 1}(\R^{\dim})^l$-valued process of individuals, each of which
may be marked. 
The dynamics are described as follows.

The process is started with a single individual 
at the point $x$ and we write
$(\xi^n_1(t),\ldots , \xi^n_{N(t)}(t))$ for the locations of the random number
$N(t)$ of individuals
at time $t$.

At time zero, independently of all else, the individual $\xi^n_1(0)$ is
marked with probability $u_n$.

At each event $(t,x,r)\in \Pi^n$, independently, 
the event is said to be neutral with probability $1-\v{s}_n$. In this case: 
\begin{enumerate}
\item if at least one individual $\xi^n_i(t-)\in \mc{B}_r(x)$ is marked,
then all marked individuals in $\mc{B}_r(x)$ are replaced by a single
offspring individual, whose location is drawn uniformly at random 
from within $\mc{B}_r(x)$;
\item for each $\xi^n_i(t)\in \mc{B}_r(x)$, including the offspring
individual if any, independently mark 
the corresponding 
individual with probability $u_n$ and unmark it otherwise.
\end{enumerate}
With the complementary probability $\v{s}_n$, the event is said to be 
selective, in which case:
\begin{enumerate}
\item if at least one individual $\xi^n_i(t-)\in \mc{B}_r(x)$ is marked,
the collection of marked individuals in $\mc{B}_r(x)$ is replaced by 
{\em three} offspring
individuals, whose locations are drawn independently and uniformly from
within $\mc{B}_r(x)$;
\item for each $\xi^n_i(t)\in \mc{B}_r(x)$, including the offspring individuals
if any,
independently mark the corresponding
individual with probability $u_n$ and unmark it otherwise.
\end{enumerate}
In between events in $\Pi_n$, nothing happens. In particular, once marked, an 
individual remains marked until it is in the region covered by an event, and,
during events, all individuals in the affected region
(whether they were marked before the event of not) sample afresh from 
independent Bernoulli random variables to decide whether they are marked 
immediately after the event.
\end{defn}
In the same way as we defined $\Xi^n$, ignoring marks, 
we write $\Phi^n$ for the
historical process corresponding to the pre-emptive dual.
The distribution of $\Phi^n$ is equal to that of $\Xi^n$. 
The only difference between Definition~\ref{def:slfvs_dual}
and Definition~\ref{def:slfvs_dual_preemptive} is that, 
for each reproduction event, 
whether or not a individual that lies in the affected region is marked
for reproduction was determined at the time of the previous reproduction
event that affected a region in which it lies. 
Notice that for both neutral and selective events,
even if no individual is marked at time $t-$, all individuals
in $\mc{B}_r(x)$ at time $t$ (after the reproduction has taken place),
independently, renew their status as marked or unmarked.

The key observation that will allow us to couple $\Xi^n$ (or 
equivalently $\Phi^n$) and $\Psi^n$ is that
for as long as two ancestral lineages are not both marked, 
they evolve independently. 

\begin{lemma}\label{lem:no_nearby} Let $T^*\in(0,\infty)$. 
There exists $\alpha>0$ such that 
$$\P\l[\exists \xi^n_{\v i} \neq \xi^n_{\v j}\subseteq\Phi^n(T^*), t\in[0,T^*]
\text{ such that }\xi^n_{\v i}\text{ and }\xi^n_{\v j}\text{ are both marked at time }t\r]
=\mc{O}(n^{-\alpha}). $$
\end{lemma}
\begin{proof}
Write $\mc T (\Phi^n(t))$ for the genealogy of $\Phi^n(t)$.
We begin by showing that for any constant $b>0$, 
$\mathcal T (\Phi ^n(T^*)) \subseteq \mathcal T^{\text{reg}}_{b\log n}$ 
with high probability.
Recall from~\eqref{eq:sel_rate_slfvs} that the rate at which each 
lineage is affected by reproduction events is 
$\eta \epsilon_n^{-2}=o(\log n)$.
Let $M^n$ be a Poisson distributed random variable with mean 
$T^* \eta \epsilon_n^{-2}$. 
Recall that if $Z'$ is Poisson with parameter $\chi$, then 
(using a Chernoff bound) for $k>\chi$,
\begin{equation}
\label{poisson tail}
\P[Z'>k]\leq \frac{e^{-\chi}(e\chi)^k}{k^k}.
\end{equation} 
Hence for $b>0$ a constant, taking $n$ sufficiently large that
$\tfrac{e\chi}{b\log n}\leq 3^{-2}$, applying~\eqref{poisson tail} with 
$k=b\log n$ and $\chi=T^* \eta \epsilon_n^{-2}=o(\log n)$, we have
$$\P \l[ M^n> b \log n \r] \leq 3^{-2b \log n}. $$
Then by a union bound over each root to leaf ray of 
$\mathcal T^{\text{reg}}_{b\log n}$,
\begin{equation}
\label{tree not containing regular tree}
\P \l[ \mathcal T (\Phi ^n(T^*)) \nsubseteq \mathcal T^{\text{reg}}_{b\log n} \r] \leq 3^{b\log n} \P \l[ M^n > b\log n\r]\leq 3^{-b\log n}.
\end{equation}

Given a particular pair of lineages, $\xi^n_{\v i},\xi^n_{\v j}\subseteq\Phi^n(t)$, 
we want to bound 
above the probability that a reproduction event occurs during $[0,T^*]$ 
after which both are marked.
The first time that this happens, at least one of $\xi^n_{\v i}$ and $\xi^n_{\v j}$
must be in the region affected by the event. After the event, the 
probability that both lineages are marked is $u_n^2$ (irrespective
of whether the second lineage was also in the affected region).
The number of reproduction events before time $T^*$ with region 
containing $\xi^{n}_{\v i}$ is Poisson with mean  $\Theta(n)$.  
Hence, the probability that a given pair $\xi^n_{\v i}, \xi^n_{\v j}$ are both marked at some time $t\in [0,T^*]$ is $\mc{O}(nu_n^2)=\mc{O}(n^{4\beta-1})$.  

Using a union bound over pairs of lineages, we have
\begin{align*}
&\P\l[\exists \xi^n_{\v i} \neq \xi^n_{\v j}\subseteq \Phi^n(T^*)\mbox{ and } 
t\in[0,T^*]\text{ such that }\xi^n_{\v i}\text{ and }\xi^n_{\v j}\text{ are both marked at time }t\r]\\
&\hspace{3pc}\leq 3^{-b\log n} + 3^{2b\log n} \mc{O}(n^{4\beta-1})\\
&\hspace{3pc}\leq 3^{-b\log n}+\mc O\l(\exp\big(2b(\log 3)(\log n)+(4\beta-1)\log n\big)\r).
\end{align*}
Noting that $4\beta-1<0$ and choosing $b$ such that $2b(\log 3)+(4\beta-1)<0$ gives the required result.
\end{proof}

\begin{proof}[Of Lemma~\ref{lem:P=D}.]
Let 
$$\tau=\inf\{t\geq 0:\exists \xi^n_{\v i} \neq \xi^n_{\v j}\subseteq\Phi^n(T^*) 
\text{ such that }\xi^n_{\v i}\text{ and }\xi^n_{\v j}\text{ are both marked at time }t\}.$$ 
Noting that for any $k\in\N$ and any 
$\alpha>0$ we have $n^{-\alpha}=o((\log n)^{-k/2})=o(\epsilon_n^k)$, 
by Lemma~\ref{lem:no_nearby}, $\P[\tau\geq T^*]\geq 1-\epsilon_n^k$. 
For as long as ancestral lineages in $\Phi^n$ are not both marked they 
evolve independently, so we may couple $(\Phi^n(t))$ and $(\Psi^n(t))$ to 
be equal up until time $\tau$ and the result follows. 
\end{proof}

\subsubsection{Generation of the interface}
\label{sec:slfvs_generation}

In this section we show that, in analogy to Proposition~\ref{prop:d=2_generation}, the interface is generated in time of order $\epsilon_n ^2|\log \epsilon_n|$. 
The proof is similar to that of Proposition~\ref{prop:d=2_generation}.

\begin{prop} \label{prop:d=2_generation_slfvs}
Let $k\in\N$. Then there exist $n_*(k),a_*(k),d_*(k)>0$ such that, for all $n\geq n_*$, 
if we set
\begin{equation}\label{eq:delta_slfv}
\delta_*(k,n) := a_*(k)\epsilon_n^2|\log\epsilon_n| 
\textrm{ and }\quad \delta'_*(k,n) 
:= (a_*(k)+\eta^{-1}(k+1))\epsilon_n^2|\log\epsilon_n|,
\end{equation}
then for $t\in [\delta_*, \delta'_*]$,
\begin{enumerate}
\item for $x$ such that $d(x,\sigma^2 t)\geq d_*\epsilon |\log \epsilon|$, we have $\P_x\l[\Vote_p (\Xi^n(t))=1\r]\geq 1-\epsilon_n^k$;
\item for $x$ such that $d(x,\sigma^2 t)\leq -d_*\epsilon |\log \epsilon|$, we have $\P_x\l[\Vote_p (\Xi^n(t))=1\r]\leq \epsilon_n^k$.
\end{enumerate}
\end{prop}
Using the coupling from Lemma~\ref{lem:P=D},
it suffices to prove the result for the branching jump process
$\Psi^n(t)$ in place of $\Xi^n(t)$. For this we exploit the following
lemma.
\begin{lemma} \label{lem:within_tree_slfvs} 
Let $k\in\N$ and 
let $A(k)$ be chosen as in Lemma~\ref{g_iteration}.
There exist $a_*(k),B_*(k)\in(0,\infty)$, and $n_*(k)<\infty$ such that for all $n\geq n_*$ and $\delta_*$, $\delta'_*$ as defined in~\eqref{eq:delta_slfv},
\begin{align} 
\P\l[\mathcal T (\Psi ^n(\delta_*)) \supseteq \mathcal T^{\text{reg}}_{A(k)|\log \epsilon_n |}\r]
&\geq 1-\epsilon_n ^k,\label{eq:D_contains_tree}\\
\text{ and }\hspace{1cm}\P\l[\mathcal T (\Psi ^n(\delta'_*)) \subseteq \mathcal T^{\text{reg}}_{B_*(k)|\log \epsilon_n|}\r]
&\geq 1-\epsilon_n ^k.\label{eq:D_in_big_tree}
\end{align}
\end{lemma}
\begin{remark}
During the proof of Proposition~\ref{prop:d=2_generation}, we 
deduced~\eqref{eq:ternary_tree_2d}, which is the equivalent 
of~\eqref{eq:D_contains_tree}. 
We did not require an equivalent of~\eqref{eq:D_in_big_tree}. We 
shall use~\eqref{eq:D_in_big_tree} here in order to prove the equivalent 
of~\eqref{eq:all_particles_bound_2d}.
\end{remark}
\begin{proof}
Recall from~\eqref{eq:sel_rate_slfvs} that a given ancestral lineage 
in $\Psi^n$ branches into three after an exponential time with 
rate $\eta \epsilon_n^{-2}$. Hence, \eqref{eq:D_contains_tree} follows for $a_*$ sufficiently large by the same proof as Lemma~\ref{ternary_tree}.

The proof of~\eqref{eq:D_in_big_tree} is the same as 
that of~\eqref{tree not containing regular tree}.
Let $L^n$ be a Poisson distributed random variable with mean 
$\delta'_* \eta \epsilon_n^{-2}=(a_*+\eta^{-1}(k+1)) \eta |\log \epsilon_n |$. 
Take $B_*=B_*(k)$ sufficiently large that $B_* \geq (a_*+\eta^{-1}(k+1)) \eta$ 
and  
\begin{equation}\label{eq:a*B*}
e(a_*+\eta^{-1}(k+1))\eta B_*^{-1}<\frac{1}{3} e^{-k/B_*}.
\end{equation}
The Chernoff bound~\eqref{poisson tail} gives
\begin{align}
\P\l[L^n > B_*|\log \epsilon_n | \r]
&\leq \l(e(a_*+\eta^{-1}(k+1))\eta B_*^{-1}\r)^{B_*|\log \epsilon_n|}\notag\\
&\leq\epsilon^k 3^{-B_*|\log \epsilon_n|},\label{eq:Mxbranching}
\end{align}
and, taking a union bound over each root to leaf ray of 
$\mathcal T^{\text{reg}}_{B_* |\log \epsilon_n|}$,
$$
\P\l[\mathcal T (\Psi ^n(\delta'_*)) \nsubseteq \mathcal T^{\text{reg}}_{B_*(k)|\log \epsilon_n|}\r] \leq 3^{B_* |\log \epsilon _n|} \P\l[L^n > B_*|\log \epsilon_n | \r]\leq \epsilon_n ^k, $$
which completes the proof.
\end{proof}

\begin{proof}[Of Proposition~\ref{prop:d=2_generation_slfvs}.]
We prove this result with $\Psi^n$ in place of $\Xi^n$ (from which the 
result follows using Lemma~\ref{lem:P=D}). The approach closely follows
that of Proposition~\ref{prop:d=2_generation}
except that now we have to control the distance between the jump process
followed by a lineage and Brownian motion.

Take $a_*$ from Lemma~\ref{lem:within_tree_slfvs}, and $t\in [\delta_*, \delta'_*]$.
Let $(\xi^n(t))_{t \geq 0}$ be a pure jump process with rate of jumps from $y$ to $y+z$ given by the intensity measure $m^n(dz)$.
By Lemma~\ref{lem:W_xin_close} we can couple $(\xi^n(t))_{t \geq 0}$ with a $\dim$-dimensional Brownian motion $(W(t))_{t\geq 0}$ in such a way that $\xi^n(0)=W(0)$ and 
$$
\P \l[ |\xi^n(t)-W(\sigma^2 t)| \geq n^{-\beta/6} \r]=\mc O(n^{-\beta }).
$$
For $d_*(k)$ a constant, for large enough $n$, since $\epsilon_n^{-2}=o(\log n)$ we have 
$\frac{1}{2}d_*\epsilon_n|\log\epsilon_n|\geq 2n^{-\beta/6}$. Hence, 
for such $n$, 
\begin{align*}
\P\l[|\xi^n(t)-\xi^n(0)| \geq \tfrac{1}{2}d_* \epsilon_n |\log \epsilon_n |\r]
&\leq \P \l[ |\xi^n(t)-W(\sigma^2 t)| \geq n^{-\beta/6} \r]\\
&\hspace{1cm}+
\P\l[|W(\sigma^2\delta'_*(k,n)))-W(0)| \geq \tfrac{1}{4}d_* \epsilon_n |\log \epsilon_n |\r]\\
&\leq \mc{O}(n^{-\beta})+2\dim\exp\l(-\frac{1}{64}
\frac{d_*^2}{\sigma^2 (a_*+\eta^{-1}(k+1))}|\log \epsilon_n |\r)\\
&\leq 3^{-B_*|\log \epsilon_n|} \epsilon_n^k.
\end{align*}
Here the second inequality follows by bounding the modulus of a $\dim$-dimensional Brownian motion by the sum of the moduli of $\dim$ one-dimensional Brownian motions, and the last inequality follows for $d_*$ sufficiently large.
Using~\eqref{eq:D_in_big_tree} and taking a union bound over the 
root to leaf rays of $\mc{T}_{B_*|\log\epsilon_n|}$, 
for $t\in[\delta_*,\delta'_*]$,
\begin{align} 
\P_x\l[\exists \xi^n_{\v i}\subseteq \Psi^n(\delta'_*) \text{ s.t. }|\xi^n_{\v i}(t)-x| 
\geq \tfrac{1}{2}d_* \epsilon_n |\log \epsilon_n |\r]
&\leq \epsilon_n^k+3^{B_*|\log \epsilon_n|}  3^{-B_*|\log \epsilon_n|} \epsilon_n^k\notag\\
&\leq 2\epsilon_n^k. \label{eq:no_moving_far_slfvs}
\end{align}
Combining~\eqref{eq:no_moving_far_slfvs} with 
Lemma~\ref{lem:within_tree_slfvs},
we obtain that, with probability $\geq 1-3 \epsilon_n^k$,
\begin{enumerate}
\item $\Vote_p (\Psi^n(t))$ is given by independent votes at each of the leaves of $\mathcal T(\Psi^n(t))$.
\item $\mathcal T (\Psi^n(t)) \supseteq \mathcal T^{\text{reg}}_{A|\log \epsilon_n|}$ and the positions of the individuals corresponding to the leaves of $\mathcal T (\Psi^n(t) )$ are all within $\tfrac{1}{2}d_* \epsilon_n |\log \epsilon_n |$ of their starting position. 
\end{enumerate} 
Just as in the proof of Proposition~\ref{prop:d=2_generation}
we obtain Proposition~\ref{prop:d=2_generation_slfvs}
with $\Psi^n$ in place of $\Xi^n$.
An application of Lemma~\ref{lem:P=D} completes the proof. 
\end{proof}

\subsubsection{Propagation of the interface}
\label{sec:slfvs_propagation}

We require the following slight modification of Lemma~\ref{lem:keylemma_2}.
\begin{lemma}\label{lem:keylemma_slfvs}
Let $l \in \N$ with $l\geq 4$ and $K_1>0$. There exists $K_2=K_2(K_1,l)>0$ 
and $n_*(l,K_1,K_2)>0$ such that for all $n\geq n_*$, $x\in\R^{\dim}$, 
$s\in [\sigma^2 \epsilon_n^{l+3},\sigma^2(l+1)\eta^{-1}\epsilon_n^2 |\log 
\epsilon_n |]$ and $t\in[s,\sigma^2 T^*]$,
\begin{align}
&E_x\l[g\l(\P^{\epsilon_n}_{d(W_{s},t-s)+K_1e^{K_2 (t-s)}\epsilon_n |\log \epsilon_n |+3n^{-\beta/6}}[\Vote(\v B(t-s))=1]+\epsilon_n ^l\r) \r]\notag\\
&\hspace{8pc}\leq \tfrac{3}{4}\epsilon_n^l+E_{d(x,t)}\l[g\l(\P^{\epsilon_n}_{B_s +K_1 e^{K_2 t}\epsilon_n |\log \epsilon_n |}[\Vote(\v B(t-s))=1]\r)\r]
+\1_{s\leq \epsilon_n ^3} \epsilon_n ^l, \label{eq:keylemma_slfvs}
\end{align}
and
\begin{align}
&E_x\l[g\l(\P^{\epsilon_n}_{d(W_{s},t-s)-K_1e^{K_2 (t-s)}\epsilon_n |\log \epsilon_n |-3n^{-\beta/6}}[\Vote(\v B(t-s))=0]+\epsilon_n ^l\r) \r]\notag\\
&\hspace{8pc}\leq \tfrac{3}{4}\epsilon_n^l+E_{d(x,t)}\l[g\l(\P^{\epsilon_n}_{B_s -K_1 e^{K_2 t}\epsilon_n |\log \epsilon_n |}[\Vote(\v B(t-s))=0]\r)\r]
+\1_{s\leq \epsilon_n ^3} \epsilon_n ^l. \label{eq:keylemma_slfvs_opp}
\end{align}
\end{lemma}
\begin{proof}
The proof is essentially the same as that of Lemma~\ref{lem:keylemma_2}.
Let $R=2 c_1(l)+4\sigma^2\eta^{-1} (l+1)\dim +1$ and fix $K_2$ such that
$
K_1(K_2-C_0)-C_0 R=2c_1(1);
$
let 
$$
A_x=\l\{\sup_{u\in [0,s]} |W_u-x|\leq 2\sigma^2\eta^{-1} (l+1)\,\dim \epsilon |\log \epsilon |\r\}.
$$
The proof for $d(x,t)\geq (2c_1(l)+2(l+1) \dim +K_1 e^{K_2 (t-s)})\epsilon_n |\log \epsilon_n|$ is then the same as in the proof of Lemma~\ref{lem:keylemma_2} (since $n^{-\beta/6} =o(\epsilon_n |\log \epsilon_n|)$).

Since $n^{-\beta/6} =o(s\epsilon_n |\log \epsilon_n|)$, we have for $\beta = (R+K_1 e^{K_2(t-s)})\epsilon |\log \epsilon |$
as in~\eqref{eq:beta_defn}, for $n$ sufficiently large
\begin{equation} \label{eq:K2_conseq_2}
K_1 e^{K_2 t} \epsilon_n |\log \epsilon_n |-(C_0\beta s +K_1 e^{K_2 (t-s)}\epsilon_n |\log \epsilon_n |+3n^{-\beta/6})
\geq c_1(1) s \epsilon_n |\log \epsilon_n |.
\end{equation}
Using~\eqref{eq:K2_conseq_2} in place of~\eqref{eq:K2_conseq}, 
the proof for $ |d(x,t)|\leq (2c_1(l)+2\sigma^2\eta^{-1}(l+1)\dim + 
K_1 e^{K_2(t-s)})\epsilon_n |\log \epsilon_n |$ is the same as in the 
proof of Lemma~\ref{lem:keylemma_2}.
\end{proof}

The equivalent of Proposition~\ref{prop:contra} for $\Psi^n$ is as follows.
\begin{prop} \label{prop:contra_slfv}
Let $l\in\N$ with $l\geq 4$. 
Define $a_*(l)$ and $\delta_* (l,n)$ as in Proposition~\ref{prop:d=2_generation_slfvs}.
There exist $K_1(l),K_2(l)>0$ and $n_*(l, K_1, K_2)>0$ such that for all 
$n\geq n_*$ and $t\in[\delta_*(l,n),T^*]$ we have 
\begin{equation}\label{eq:upper_final_slfv}
\sup\limits_{x\in\R^{\dim}}\Big(\P_x\l[\Vote_p(\v \Psi^n (t))=1\r]-\P^{\epsilon_n}_{d(x,\sigma^2 t)+K_1e^{K_2 \sigma^2 t}\epsilon_n|\log\epsilon_n|}\l[\Vote(\v B(\sigma^2 t))=1\r]\Big)\leq \epsilon_n^l
\end{equation}
and
\begin{equation} \label{eq:lower_final_slfvs}
\sup\limits_{x\in\R^{\dim}}\Big(\P_x\l[\Vote_p(\v \Psi^n (t))=0\r]-\P^{\epsilon_n}_{d(x,\sigma^2 t)-K_1e^{K_2 \sigma^2 t}\epsilon_n|\log\epsilon_n|}\l[\Vote(\v B(\sigma^2 t))=0\r]\Big)\leq \epsilon_n^l.
\end{equation}
\end{prop}
\begin{proof}
The proof exactly follows that of Proposition~\ref{prop:contra}, with
Corollary~\ref{cor:xi_W} and then Lemma~\ref{lem:keylemma_slfvs} 
in place of Lemma~\ref{lem:keylemma_2},
and Proposition~\ref{prop:d=2_generation_slfvs} 
in place of Proposition~\ref{prop:d=2_generation}. 
\end{proof}

\begin{proof}[Of Theorem~\ref{thm:slfvs_dual}]
It suffices to prove the result for sufficiently large $k\in\N$, and in particular we will show it for $k\geq 5$. By Lemma~\ref{lem:P=D}, for $n$ sufficiently large and $t\in [0,T^*]$,
\begin{equation*} 
|\P_x\l[\Vote_p(\v \Psi^n (t))=1\r]-\P_x\l[\Vote_p(\v \Xi^n (t))=1\r]|\leq \epsilon_n^{k+1}.
\end{equation*}
The result now follows from Proposition \ref{prop:contra_slfv} with $l=k+1$, in the same way as in the proof of Theorem~\ref{thm:BBMtwo}.
\end{proof}

\bibliographystyle{plainnat}
\bibliography{curvature}

\end{document}